\documentclass[a4paper,11pt]{article}

\usepackage[T1]{fontenc}
\usepackage{lmodern}
\usepackage[
    top=1.8cm,
    bottom=2.2cm,
    left=2.4cm,
    right=2.4cm
]{geometry}

\usepackage{amsmath,amssymb,amsthm}
\usepackage{cite}
\usepackage{enumitem}
\usepackage{array}
\usepackage{tabularx}
\usepackage{xcolor}
\usepackage{float}

\usepackage{hyperref}

\hypersetup{
    colorlinks=true,
    linkcolor=blue,
    citecolor=blue,
    urlcolor=blue,
    pdftitle={
        Two-point modulus for the logarithmic p-flux of the first Dirichlet eigenfunction of the p-Laplacian
    },
    pdfauthor={
        Rui Chen
    }
}

\numberwithin{equation}{section}

\theoremstyle{plain}
\newtheorem{theorem}{Theorem}[section]
\newtheorem{lemma}[theorem]{Lemma}
\newtheorem{proposition}[theorem]{Proposition}
\newtheorem{corollary}[theorem]{Corollary}

\theoremstyle{definition}

\newtheorem{notation}{Notation}[section]

\newcommand{\R}{\mathbb{R}}

\renewcommand{\phi}{\varphi}

\newcolumntype{Y}{>{\centering\arraybackslash}X}

\allowdisplaybreaks

\begin{document}

\begin{center}
\vspace*{-0.7cm}

{\large\bfseries
Two-point estimates for the logarithmic \(p\)-flux of the first Dirichlet \(p\)-eigenfunction
\par}

\vspace{0.7cm}

{\normalsize

Rui Chen

\par}
\end{center}

\vspace{0.1cm}


\begin{abstract}
Let \(u>0\) be the first Dirichlet \(p\)-eigenfunction on a bounded convex domain \(\Omega\subset\mathbb R^N\), set
\[
X_\Omega:=|\nabla\log u|^{p-2}\nabla\log u.
\]
We study whether the sharp one-dimensional two-point modulus for \(X_\Omega\), which for \(p=2\) reduces to the logarithmic-gradient estimate of Andrews and Clutterbuck, persists for \(p\neq2\). We prove sharp estimates on intervals and balls for every \(p>1\), with the radial modulus on balls strictly larger than the one-dimensional one. In dimensions \(N\ge2\), however, the one-dimensional modulus fails on general convex domains for every \(p\neq2\). For \(1<p<2\), at every sufficiently small fixed scale there are smooth uniformly convex domains for which the corresponding two-point flux tends to zero. For \(p>2\), on thin domains \(\Omega_\varepsilon=D\times(-\varepsilon,\varepsilon)\), with \(D\subset\mathbb R^{N-1}\) bounded and convex, the normalized first eigenfunctions satisfy
\[
u_\varepsilon(x,\varepsilon z)\longrightarrow
\frac{\phi(z)}{\phi(0)}
\left(\frac{G_D(x)}{G_D(x_0)}\right)^{2/p}
\]
locally uniformly in \(D\times(-1,1)\), where \(\phi\) and \(G_D\) are the first Dirichlet eigenfunctions of the \(p\)-Laplacian on \((-1,1)\) and of the Laplacian on \(D\), respectively. This yields the failure for \(p>2\). Finally, for arbitrary \(C^2\) functions, positivity of the symmetric differential of the \(p\)-gradient implies convexity, and the converse holds universally if and only if \(p=2\).
\end{abstract}

\medskip


\medskip


\bigskip

\tableofcontents

\bigskip


\section{Introduction and Main Results}
\label{sec:introduction}

In this paper, we study logarithmic concavity and two-point estimates for the positive first Dirichlet eigenfunction of the \(p\)-Laplacian on convex domains. Let \(1<p<\infty\) and let \(\Omega\subset\R^N\) be a bounded domain. We consider
\begin{equation}\label{eq:mainequation}
\begin{cases}
-\Delta_p u
=
\lambda_{1,p}(\Omega)u^{p-1}
& \text{in }\Omega,\\
u=0
& \text{on }\partial\Omega,
\end{cases}
\qquad
\Delta_p u
:=
\operatorname{div}\!\left(|\nabla u|^{p-2}\nabla u\right),
\end{equation}
where
\begin{equation}\label{eq:intro-lambda1}
\lambda_{1,p}(\Omega)
:=
\inf_{u\in W_0^{1,p}(\Omega)\setminus\{0\}}
\frac{\displaystyle\int_\Omega |\nabla u|^p\,dx}
{\displaystyle\int_\Omega |u|^p\,dx}
\end{equation}
is the first Dirichlet eigenvalue. A nonnegative minimizer in
\eqref{eq:intro-lambda1} may be chosen strictly positive in \(\Omega\)
and is unique up to multiplication by a constant. The simplicity and
positivity of the first eigenfunction are classical; see, for instance,
Anane~\cite{Anane1987}, Lindqvist~\cite{Lindqvist1990}, and Otani and
Teshima~\cite{otani-teshima}. For balls, radial symmetry of the first
eigenfunction was proved by Bhattacharya~\cite[Theorem~1]{Bhattacharya1988}.

\smallskip

Convexity properties of positive first eigenfunctions have been studied
since the work of Brascamp and Lieb~\cite{BrascampLieb1976}, who proved
that the first Dirichlet eigenfunction of the Laplacian is log-concave
on convex domains. Related PDE methods and concavity principles were
developed by Caffarelli and Spruck~\cite{CaffarelliSpruck1982},
Korevaar~\cite{Korevaar1983}, and Caffarelli and
Friedman~\cite{CaffarelliFriedman1985}. For the \(p\)-Laplacian,
Sakaguchi~\cite{Sakaguchi1987} proved the log-concavity of the positive
first Dirichlet eigenfunction on smooth convex domains. This result was
later extended to an anisotropic setting by Kawohl and
Novaga~\cite{KawohlNovaga2008}. More general concavity results, after
suitable transformations of the solution, for quasilinear
\(p\)-Laplace equations were obtained by Borrelli, Mosconi, and
Squassina~\cite{BorrelliMosconiSquassina2024}. Crasta and Fragal\`a~\cite{CrastaFragala2020} studied the existence and
log-concavity of positive principal eigenfunctions for a class of
fully nonlinear homogeneous elliptic operators that includes the
\(p\)-Laplacian.

In the linear case, the qualitative log-concavity theorem admits a much
sharper form. Andrews and Clutterbuck~\cite{AndrewsClutterbuck2011}
proved a sharp two-point modulus of concavity for \(\log u\) on convex
domains, and used it to prove the fundamental gap conjecture. Their
estimate compares the higher-dimensional first eigenfunction with the
one-dimensional eigenfunction on an interval of the same diameter.
Related modulus estimates for vector fields satisfying nonlinear
equations were developed by Ni~\cite{Ni2013}. More recently, log-concavity and related Brunn--Minkowski inequalities
have been studied for weighted \(p\)-Laplace operators by
Qin~\cite{Qin2024}, and for the Gaussian \(p\)-operator by Colesanti,
Qin, and Salani~\cite{ColesantiQinSalani2026}. Bryan, Clutterbuck, and
Rankin~\cite{BryanClutterbuckRankin2026} recently gave a new proof of
the log-concavity of the first eigenfunction based on a sup-convolution
argument, deriving it as a consequence of a Brunn--Minkowski inequality
for the first Dirichlet eigenvalue.

\smallskip

For \(p\neq2\), the logarithmic gradient itself is no longer the quantity naturally associated with the \(p\)-Laplacian. Since
\[
|\nabla u|^{p-2}\nabla u
=
u^{p-1}
|\nabla\log u|^{p-2}\nabla\log u,
\]
we introduce the logarithmic \(p\)-flux $X_\Omega
:=
|\nabla\log u|^{p-2}\nabla\log u.$
Indeed, dividing \eqref{eq:mainequation} by \(u^{p-1}\) gives
\[
\operatorname{div}X_\Omega
+
(p-1)|\nabla\log u|^p
=
-\lambda_{1,p}(\Omega).
\]
Thus \(X_\Omega\) is the nonlinear analogue of
\(\nabla\log u\) in the linear case. When \(p=2\), $X_\Omega=\nabla\log u,$
which is precisely the quantity appearing in the
Andrews--Clutterbuck modulus of concavity. This motivates asking whether,
for \(p\neq2\), \(X_\Omega\) satisfies an analogous sharp two-point
estimate determined by the one-dimensional \(p\)-eigenfunction.

\smallskip

We first give an affirmative answer in dimension one. More precisely, we establish the sharp one-dimensional two-point modulus estimate for every $1<p<\infty$. We begin by fixing the notation for the one-dimensional model.

\begin{notation}
\label{not:one-dimensional-model}
Let \(I:=(-R,R)\), and let \(u_I>0\) be the first eigenfunction of
\eqref{eq:mainequation} on \(I\), with corresponding eigenvalue
\(\lambda_{1,p}(I)\), normalized by \(u_I(0)=1\). Then \(u_I\) is even,
\(u_I(\pm R)=0\), \(u_I'(0)=0\), and \(u_I'(x)<0\) for \(0<x<R\). Writing \(u_I(x)=U_I(r)\), where \(r=|x|\), define
\[
q_I(r):=-\frac{U_I'(r)}{U_I(r)},
\qquad
Q_I(r):=q_I(r)^{p-1},
\quad 0\le r<R.
\]
Then \(q_I(0)=Q_I(0)=0\), while \(q_I,Q_I>0\) on \((0,R)\). Setting
\(p'=p/(p-1)\), we have
\begin{equation}\label{eq:QI-ode-main} 
    Q_I'(r)
=
\lambda_{1,p}(I)+(p-1)Q_I(r)^{p'},
\quad 0<r<R.
\end{equation}
Consequently, \(q_I\) and \(Q_I\) are strictly increasing and tend to $+\infty$ as $r\uparrow R$.
Finally, define the logarithmic \(p\)-flux by
\[
X_I(x)
:=
\left|(\log u_I)'(x)\right|^{p-2}(\log u_I)'(x)
=
-\operatorname{sgn}(x)\,Q_I(|x|),
\quad x\in I,
\]
where \(\operatorname{sgn}(0)=0\).
\end{notation}

We first establish the one-dimensional two-point modulus estimate. The key observation is that \eqref{eq:QI-ode-main} implies that \(-X_I'\) is even and strictly increasing in \(|x|\). Consequently, its integral over any interval of fixed length is minimized when the interval is centered at the origin.

\begin{theorem}
\label{thm:one-dimensional-two-point-estimate}
Let \(1<p<\infty\), and adopt the notation of
Notation~\ref{not:one-dimensional-model}. Then, for every distinct
\(x,y\in I\),
\[
\left(
X_I(y)-X_I(x)
\right)\frac{y-x}{|y-x|}
\le
-2Q_I\!\left(\frac{|x-y|}{2}\right).
\]
Moreover, equality holds if and only if \(x=-y\).
\end{theorem}

We next turn to the higher-dimensional case. We first establish a positive result for balls. For convenience, we begin by fixing the notation for the ball model.

\begin{notation}
\label{not:ball-model}
Let \(N\ge2\), \(1<p<\infty\), and \(B_R:=B_R(0)\). Let \(u_B>0\) be the first eigenfunction of
\eqref{eq:mainequation} on \(B_R\) with eigenvalue
\(\lambda_{1,p}(B_R)\), \(u_B(0)=1\). Writing $u_B(x)=U_B(|x|),$
we have
\[
U_B(0)=1,\qquad U_B'(0)=0,\qquad U_B(R)=0,\qquad U_B'(r)<0\quad \text{when}\quad 0<r<R.
\]
Define
\[
q_B(r):=-\frac{U_B'(r)}{U_B(r)},
\qquad
Q_B(r):=q_B(r)^{p-1},
\quad 0\le r<R.
\]
Then \(q_B(0)=Q_B(0)=0\), while \(q_B,Q_B>0\) on \((0,R)\). The logarithmic \(p\)-flux is
\[
X_B(x)
:=
|\nabla\log u_B(x)|^{p-2}\nabla\log u_B(x)
=
-Q_B(|x|)\frac{x}{|x|}
\quad\text{for }x\ne0,
\]
with \(X_B(0)=0\).
\end{notation}

The next result gives the sharp two-point estimate on balls.

\begin{theorem}\label{thm:intro-comparison}
Let \(N\ge2\), \(1<p<\infty\), and \(R>0\). Then, for every distinct
\(x,y\in B_R\),
\[
\left\langle
X_B(y)-X_B(x),\frac{y-x}{|y-x|}
\right\rangle
\le
-2Q_B\!\left(\frac{|x-y|}{2}\right),
\]
with equality if and only if \(y=-x\). Moreover,
\[
Q_B(r)>Q_I(r)
\qquad\text{for every }0<r<R.
\]
\end{theorem}

Hence \(Q_B\) is the sharp modulus on the ball, while the
one-dimensional modulus \(Q_I\) is strictly weaker. In particular,
\[
\left\langle
X_B(y)-X_B(x),\frac{y-x}{|y-x|}
\right\rangle
<
-2Q_I\!\left(\frac{|x-y|}{2}\right)
\]
for every distinct \(x,y\in B_R\). Equality in the sharp estimate with
\(Q_B\) occurs exactly for antipodal pairs \(y=-x\), whereas the
comparison with \(Q_I\) remains strict.

\smallskip

The proof has two parts. We first use the radial form of \(X_B\), together
with the monotonicity and convexity of \(Q_B\), to prove the sharp
two-point estimate with \(Q_B\). We then compare \(Q_B\) with \(Q_I\).
The asymptotic behavior at the origin and the boundary, together with a
contact-point argument for the corresponding first-order equations, gives
\(Q_B>Q_I\) on \((0,R)\).

\smallskip The proof of Theorem~\ref{thm:intro-comparison} also yields the following eigenvalue comparison between the ball and the one-dimensional model.

\begin{corollary}\label{cor:ball-interval-eigenvalue}
Let \(N\ge2\), \(1<p<\infty\), and \(R>0\). Then
\[
\lambda_{1,p}(B_R)
>
N\lambda_{1,p}((-R,R))
=
N(p-1)\left(\frac{\pi_p}{2R}\right)^p.
\]
\end{corollary}

The eigenvalue comparison in Corollary~\ref{cor:ball-interval-eigenvalue}
also gives a direct improvement of a known lower bound for
\(\lambda_{1,p}(B_R)\). The one-dimensional eigenvalue and the associated
generalized trigonometric functions are classical; see, for instance,
\cite{Lindqvist1995,LangEdmunds2011,EdmundsGurkaLang2012}. Estimates for
the principal \(p\)-eigenvalue in higher dimensions have been obtained in
\cite{BenediktDrabek2012,Benedikt2015,AghajaniMoslehTehrani2018,
Kajikiya2015}. In particular, Benedikt and Dr\'abek proved $\lambda_{1,p}(B_R)\ge \frac{Np}{R^p}.$
On the other hand, by \cite[Corollary~1.4]{KajikiyaTakeuchi2025},
\[
\lambda_{1,p}((-1,1))
=(p-1)\left(\frac{\pi_p}{2}\right)^p>p,
\qquad 1<p<\infty.
\]
Hence Corollary~\ref{cor:ball-interval-eigenvalue} gives
\[
\lambda_{1,p}(B_R)
>
\frac{N}{R^p}\lambda_{1,p}((-1,1))
>
\frac{Np}{R^p}.
\]
For \(p=2\), related lower bounds follow from classical estimates for
the first positive zero of Bessel functions; see \cite{Lorch1993}.
To the best of our knowledge, the above strict ball-to-interval comparison
has not previously been recorded in this form for the Dirichlet
\(p\)-Laplacian.

\smallskip

Theorem~\ref{thm:intro-comparison} shows that the sharp one-dimensional two-point modulus persists on balls for every \(1<p<\infty\). This positive result is a consequence of the radial structure of the first eigenfunction, which reduces the logarithmic \(p\)-flux to the scalar equation
\[
Q_B'(r)+\frac{N-1}{r}Q_B(r)
=
\lambda_{1,p}(B_R)+(p-1)Q_B(r)^{p'}.
\]
Such a reduction is no longer available on general convex domains. 

For \(p\neq2\), the corresponding one-dimensional two-point modulus fails in general. The arguments are different in
the subquadratic and superquadratic cases. For \(1<p<2\), the tangential
logarithmic \(p\)-flux vanishes near a flat boundary patch, which leads,
after smooth uniformly convex approximation, to fixed-scale two-point
fluxes converging to zero. For \(p>2\), thin-domain asymptotics on
\(D\times(-\varepsilon,\varepsilon)\) show that the first longitudinal
correction is governed by the Dirichlet Laplacian on \(D\), rather than by
the one-dimensional \(p\)-Laplacian. Finally,
Section~\ref{sec:geometric-structure-logarithmic-flux} shows that, for
arbitrary \(C^2\) functions, positive semidefiniteness of the symmetric
differential of the \(p\)-gradient implies convexity, while the converse
holds for all such functions if and only if \(p=2\).

\smallskip

The following theorem shows that, for \(1<p<2\), the component of
\(X_{\Omega}(y)-X_{\Omega}(x)\) in the direction \(y-x\) can be made
arbitrarily small for every sufficiently small \(r>0\), with
\(|x-y|=2r\), even on smooth uniformly convex domains of fixed diameter.

\begin{theorem}
\label{thm:failure-universal-two-point-subquadratic}
Let \(N\ge2\), \(1<p<2\), and \(D>0\). Then there exists \(r_*>0\) such that,
for every  \(r\in(0,r_*)\), there exist bounded \(C^\infty\)
uniformly convex domains \(\Omega_k\subset\mathbb R^N\) and distinct points
\(x_k,y_k\in\Omega_k\) satisfying
\[
\operatorname{diam}(\Omega_k)=D,
\qquad
|x_k-y_k|=2r,
\]
for which, if \(u_k>0\) is the first eigenfunction of
\eqref{eq:mainequation} on \(\Omega_k\) and
\[
X_{\Omega_k}
:=
|\nabla\log u_k|^{p-2}\nabla\log u_k,
\]
then
\[
\left\langle
X_{\Omega_k}(y_k)-X_{\Omega_k}(x_k),
\frac{y_k-x_k}{|y_k-x_k|}
\right\rangle
\longrightarrow0
\qquad\text{as }k\to\infty.
\]
\end{theorem}

In particular, the two-point estimate associated with the one-dimensional model
cannot hold for all smooth uniformly convex domains.
In fact, Theorem \ref{thm:failure-universal-two-point-subquadratic} gives a stronger conclusion: no estimate of the
form
\[
\left\langle
X_\Omega(y)-X_\Omega(x),
\frac{y-x}{|y-x|}
\right\rangle
\le -2\psi\left(\frac{|x-y|}{2}\right)
\]
can hold uniformly over this class of domains if
\(\psi(r)>0\) for some sufficiently small \(r\).

The proof starts from a convex domain containing a flat boundary patch.
Writing \(t\) for the inward normal distance and \(z\) for the tangential
variable, the Hopf expansion gives
\[
u(z,t)=a(z)t+O(t^2),
\]
and hence $|\nabla\log u(z,t)|^{p-2}=O(t^{2-p}).$
Since \(1<p<2\), the tangential component of the logarithmic \(p\)-flux
therefore tends to zero as \(t\downarrow0\). For every sufficiently small fixed \(r>0\), choosing two points at the same normal height and with tangential separation \(2r\) makes the corresponding two-point flux quantity tend to zero as the normal height approaches the boundary.
We then approximate the flat domain by smooth uniformly convex domains
and rescale them so that both the diameter \(D\) and the separation
\(|x-y|=2r\) are preserved. Combining the boundary limit with the smooth uniformly convex approximation yields the sequence in Theorem~\ref{thm:failure-universal-two-point-subquadratic}.

\smallskip

For \(p>2\), the failure is of a different type. We consider thin convex
domains $\Omega_\varepsilon
:=
D\times(-\varepsilon,\varepsilon),$
where \(D\subset\mathbb R^{N-1}\) is an arbitrary bounded convex domain.
After rescaling the transverse variable, the leading-order term is given
by the first one-dimensional \(p\)-eigenvalue, whereas the second-order
tangential term is governed by the ordinary Dirichlet Laplacian on
\(D\). The tangential factor in the limiting rescaled eigenfunction is
therefore determined by the first Dirichlet eigenfunction of the
Laplacian on \(D\), rather than by a \(p\)-Laplacian eigenfunction. This
is incompatible with the one-dimensional \(p\)-Laplacian comparison.

\begin{theorem}
\label{profileprop:thin-domain}
Let \(N\geq2\), \(p>2\), \(J:=(-1,1)\), and let
\(D\subset\mathbb R^{N-1}\) be a bounded convex domain.
Fix \(x_0\in D\), and let \(u_\varepsilon>0\) be the first eigenfunction of
\eqref{eq:mainequation} on $\Omega_\varepsilon:=D\times(-\varepsilon,\varepsilon),$
normalized by \(u_\varepsilon(x_0,0)=1\). Let \(\phi>0\) be the first
eigenfunction of \eqref{eq:mainequation} on \(J\), normalized by
\[
\int_J\phi^p\,dz=1,
\qquad
A_p:=\int_J|\phi'|^{p-2}\phi^2\,dz.
\]
Let \(G_D>0\) be the first Dirichlet eigenfunction of the Laplacian on
\(D\), normalized by $\int_DG_D^2\,dx=1.$
Then
\[
u_\varepsilon(x,\varepsilon z)
\longrightarrow
\frac{\phi(z)}{\phi(0)}
\left(
\frac{G_D(x)}{G_D(x_0)}
\right)^{2/p}
\]
locally uniformly in \(D\times J\) as \(\varepsilon\downarrow0\).
Moreover,
\begin{equation}\label{eq:thin-eigenvalue-second-order}
\varepsilon^p\lambda_{1,p}(\Omega_\varepsilon)
=
\lambda_{1,p}(J)
+
\frac{2A_p}{p}\lambda_{1,2}(D)\varepsilon^2
+
o(\varepsilon^2).
\end{equation}
\end{theorem}

Theorem~\ref{profileprop:thin-domain} shows that, although the leading
transverse term is determined by the one-dimensional \(p\)-Laplacian, the
first nontrivial tangential correction is quadratic. More precisely,
after factoring out the transverse eigenfunction, the second-order energy
contains the weighted term
\[
\int_D F^{p-2}|\nabla F|^2\,dx,
\]
which becomes the ordinary Dirichlet energy under the change of variables
\(G=F^{p/2}\). This gives the first Dirichlet eigenvalue problem for the
Laplacian on \(D\), and yields both the coefficient
\(\lambda_{1,2}(D)\) in \eqref{eq:thin-eigenvalue-second-order} and the
rescaled limit of \(u_\varepsilon\).

The proof is based on a Picone-type decomposition of the anisotropically
rescaled \(p\)-energy, weighted transverse compactness, and the sharp
Dirichlet Poincar\'e inequality on \(D\). The local uniform convergence
of the rescaled eigenfunctions follows from the log-concavity of the first
\(p\)-eigenfunction. This limit will be used below to show that the
one-dimensional two-point modulus cannot hold on sufficiently thin
domains when \(p>2\).

\begin{theorem}
\label{thm:general-convex-flux-failure-p-superquadratic}
Let \(N\geq2\), \(p>2\), and let
\(D\subset\mathbb R^{N-1}\) be a bounded convex domain.
Let \(G_D>0\) be the first Dirichlet eigenfunction of the Laplacian on
\(D\), let \(x_D\in D\) be a maximum point of \(G_D\), and fix
\(e\in\mathbb S^{N-2}\). Then there exists \(t_{p,D}>0\) such that,
for every \(t\in(0,t_{p,D})\), there exists
\(\varepsilon_{p,D,t}>0\) such that, for every
\(0<\varepsilon<\varepsilon_{p,D,t}\), the points
\[
x:=(x_D,0),
\qquad
y:=(x_D+te,0)
\]
belong to $\Omega_\varepsilon
:=
D\times(-\varepsilon,\varepsilon)
\subset\mathbb R^N$
and satisfy
\[
\left\langle
X_{\Omega_\varepsilon}(y)-X_{\Omega_\varepsilon}(x),
\frac{y-x}{|y-x|}
\right\rangle
>
-2Q_{I_\varepsilon}\!\left(\frac{|x-y|}{2}\right),
\]
where
\[
X_{\Omega_\varepsilon}
:=
|\nabla\log u_\varepsilon|^{p-2}\nabla\log u_\varepsilon,
\qquad
2R_\varepsilon
:=
\operatorname{diam}(\Omega_\varepsilon),
\qquad
I_\varepsilon:=(-R_\varepsilon,R_\varepsilon),
\]
\(u_\varepsilon>0\) is the first eigenfunction of
\eqref{eq:mainequation} on \(\Omega_\varepsilon\), and
\(Q_{I_\varepsilon}\) is defined in
Notation~\ref{not:one-dimensional-model}.
\end{theorem}

The proof uses the thin-domain limit from
Theorem~\ref{profileprop:thin-domain}. After normalizing
\(u_\varepsilon(x_D,0)=1\), the rescaled eigenfunctions converge locally
uniformly, and in particular
\[
u_\varepsilon(x,0)
\longrightarrow
\left(
\frac{G_D(x)}{G_D(x_D)}
\right)^{2/p}
\]
locally uniformly in \(D\). The concavity of
\(\log u_\varepsilon(\cdot,0)\) then gives convergence of the tangential
logarithmic gradients on \(D\times\{0\}\). Since
\(\nabla G_D(x_D)=0\), the limiting logarithmic \(p\)-flux at
\(x_D+te\) is \(O(t^{p-1})\) as \(t\downarrow0\). On the other hand,
\[
2Q_{I_{R_D}}\!\left(\frac t2\right)
=
\lambda_{1,p}(I_{R_D})t+o(t).
\]
Thus, for every sufficiently small fixed \(t>0\), the limiting flux violates the one-dimensional
two-point estimate, and the same strict inequality holds on
\(\Omega_\varepsilon\) for all sufficiently small \(\varepsilon>0\).

\smallskip

We next compare the Hessian of a function \(w\) with the symmetric
differential of $|\nabla w|^{p-2}\nabla w.$
Let \(U\subset\mathbb R^N\) be open and let \(w\in C^2(U)\). Set
\[
s:=|\nabla w|,
\qquad
A:=|\nabla w|^{p-2}\nabla w,
\qquad
\mathcal R_w:=\{x\in U:\nabla w(x)\neq0\}.
\]
For \(x\in\mathcal R_w\), let
\[
\nu:=\frac{\nabla w}{s},
\qquad
\Sigma_x:=\{y\in U:w(y)=w(x)\}.
\]
Denote by \(T_x\Sigma_x\) the tangent space of \(\Sigma_x\), by
\[
\mathrm{II}(\tau,\eta)
:=
\langle D_\tau\nu,\eta\rangle,
\qquad
\tau,\eta\in T_x\Sigma_x,
\]
its second fundamental form, and by \(\nabla_\Sigma s\) the tangential
gradient of \(s\), so that
\[
\nabla s
=
\nabla_\Sigma s+(\partial_\nu s)\nu.
\]

\begin{proposition}
\label{prop:level-set-flux-decomposition}
Let \(x\in\mathcal R_w\). With respect to the orthogonal decomposition $\mathbb R^N
=
T_x\Sigma_x\oplus\operatorname{span}\{\nu\},$
one has
\[
D^2w=
\begin{pmatrix}
s\,\mathrm{II} & \nabla_\Sigma s\\
(\nabla_\Sigma s)^T & \partial_\nu s
\end{pmatrix}.
\]
Moreover, if
\[
M_p[w]
:=
\frac12\bigl(DA+(DA)^T\bigr),
\]
where \(DA\) denotes the Jacobian matrix of
\(A=|\nabla w|^{p-2}\nabla w\), then
\[
M_p[w]
=
s^{p-2}
\begin{pmatrix}
s\,\mathrm{II} & \dfrac p2\nabla_\Sigma s\\[2mm]
\dfrac p2(\nabla_\Sigma s)^T & (p-1)\partial_\nu s
\end{pmatrix}.
\]
Here \(\mathrm{II}\) is viewed as a symmetric bilinear form on
\(T_x\Sigma_x\), \(\nabla_\Sigma s\in T_x\Sigma_x\), and $\partial_\nu s:=\langle\nabla s,\nu\rangle.$
\end{proposition}

The preceding formulas give the following special cases.

\begin{corollary}
\label{cor:rigid-flux-cases}
For every \(x\in\mathcal R_w\), the following statements hold.
\begin{enumerate}[label=\textnormal{(\roman*)}]
\item If \(p=2\), then $M_2[w]=D^2w.$

\item If \(N=1\), then $M_p[w]
=
(p-1)|w'|^{p-2}w'',$
and hence, $
M_p[w]\geq0\Longleftrightarrow
w''\geq0.$

\item If \(\nabla_\Sigma s=0\), then
\[
D^2w=
\begin{pmatrix}
s\,\mathrm{II} & 0\\
0 & \partial_\nu s
\end{pmatrix},
\qquad
M_p[w]
=
s^{p-2}
\begin{pmatrix}
s\,\mathrm{II} & 0\\
0 & (p-1)\partial_\nu s
\end{pmatrix},
\]
and consequently $D^2w\geq0 \Longleftrightarrow
M_p[w]\geq0.$
\end{enumerate}
\end{corollary}

For general \(w\), positivity of \(M_p[w]\) always implies convexity at
regular points, while the converse for every \(w\) characterizes
\(p=2\).

\begin{theorem}
\label{thm:p2-flux-rigidity}
Let \(U\subset\mathbb R^N\) be open, \(N\ge2\), \(p>1\), and
\(w\in C^2(U)\). At every point \(x\in U\) with \(\nabla w(x)\neq0\),
\[
M_p[w](x)\ge0
\quad\Longrightarrow\quad
D^2w(x)\ge0.
\]
The converse implication holds for all such \(U,w,x\) if and only if
\(p=2\).
\end{theorem}

We now apply these formulas to the first eigenfunction of
\eqref{eq:mainequation}. Let \(u>0\) be the first eigenfunction on a
convex domain \(\Omega\subset\mathbb R^N\) and set $w:=-\log u.$
Then $|\nabla w|^{p-2}\nabla w=-X_\Omega.$
By the interior \(C^{1,\alpha}\)-regularity for \(p\)-Laplace equations
\cite{Tolksdorf1984}, $u\in C^{1,\alpha}_{\mathrm{loc}}(\Omega).$
On
\[
\mathcal R_w=\{x\in\Omega:\nabla w(x)\neq0\},
\]
one has \(\nabla u\neq0\), and the equation is locally uniformly
elliptic. Iterating the interior Schauder estimates
\cite[Chapter~6]{GilbargTrudinger} gives $u,w\in C^\infty_{\mathrm{loc}}(\mathcal R_w).$
Corollary~\ref{cor:rigid-flux-cases} agrees with the positive results in
\cite[Theorem~1.5]{AndrewsClutterbuck2011} and
Theorem~\ref{thm:one-dimensional-two-point-estimate}. It also applies
directly to balls. If $w(x)=W(|x-x_0|),$
then $s=|\nabla w|=|W'|$
is constant on each spherical level set, and hence $\nabla_\Sigma s=0.$

\smallskip

Finally, we mention a consequence for the fundamental gap problem. In the
linear case, the approach of Andrews and Clutterbuck first obtains a sharp
modulus of concavity for \(\log u_1\) and then applies it to the quotient
\(v=u_2/u_1\), which satisfies the linear drift equation
\[
-\Delta v-2\nabla\log u_1\cdot\nabla v
=(\lambda_{2,2}(\Omega)-\lambda_{1,2}(\Omega))v.
\]
Our results indicate that this approach does not extend directly to the
\(p\)-Laplacian when \(p\neq2\). First, the corresponding sharp
one-dimensional modulus for the logarithmic \(p\)-flux fails on general
convex domains. More basically, the nonlinear eigenvalue equation does not
admit the analogous quotient reduction: \(u_2/u_1\) no longer satisfies a
closed linear equation whose drift is determined by \(u_1\). This suggests
that the \(p\)-Laplacian fundamental gap requires a different approach; see
\cite{ChenHauer2026} for a variational treatment of this problem.

\smallskip
The main results are summarized in the following table.

\begin{table}[H]
\centering
\begin{tabular}{|c|c|c|c|}
\hline
 & \(1<p<2\) & \(p=2\) & \(p>2\) \\ \hline
\(N=1\) & Holds & Holds & Holds \\ \hline
\(N\geq2\) & Fails in general & Holds & Fails in general \\ \hline
\end{tabular}
\caption{Validity of the sharp one-dimensional two-point estimate for
\(X_\Omega\): it holds for all \(p>1\) in dimension \(N=1\), and for
\(p=2\) in dimensions \(N\ge2\); it fails in general for \(N\ge2\) and
\(p\neq2\). On balls, it holds for every \(p>1\).}
\label{tab:intro-modulus-summary}
\end{table}

The remainder of the paper is organized as follows. Section~2 establishes the sharp two-point modulus estimate on intervals and balls, covering the one-dimensional case and the higher-dimensional radial case. Section~3 proves the failure of the one-dimensional modulus in \(N\geq2\) for \(p\neq2\), with separate arguments for \(1<p<2\) and \(p>2\); the latter relies on thin-domain asymptotics for arbitrary bounded convex domains \(D\subset\mathbb R^{N-1}\). Section~4 analyzes the Hessian of \(-\log u\) and the symmetric differential of the logarithmic \(p\)-flux along level sets, and proves that \(p=2\) is the unique exponent for which the positivity of two matrices is universally equivalent.

\section{Sharp Two-Point Estimates on Intervals and Balls}

In this section, we establish the two-point modulus estimate first on the one-dimensional interval and then on higher-dimensional balls. For the interval, the estimate follows directly from the scalar equation satisfied by the logarithmic \(p\)-flux. On the ball, radial symmetry again reduces the problem to a one-dimensional flux equation, but with an additional geometric term. We derive the monotonicity and convexity properties of the radial flux in Lemma \ref{lem:Q-convexity}, obtain an intrinsic two-point estimate on the ball in Proposition \ref{prop:intrinsic}, and then compare the ball with the one-dimensional model.

\subsection{The one-dimensional sharp estimate}

We first give the proof of the one-dimensional two-point modulus estimate.

\begin{proof}[\textbf{Proof of Theorem~\ref{thm:one-dimensional-two-point-estimate}.}]
Since \(u_I\) is even and strictly decreasing on \((0,R)\), \(X_I\) is odd and
\[
X_I'
=
-\lambda_{1,p}(I)-(p-1)|X_I|^{p/(p-1)}<0 \quad \text{on}\quad(0,R). 
\]
Thus \(g(t):=-X_I'(t)\) is even and strictly increasing as a function of \(|t|\) on \((0,R)\). Assume \(x<y\), and set \(m:=\frac{x+y}{2}\) and \(r:=\frac{y-x}{2}\). Then
\[
X_I(x)-X_I(y)
=
\int_{m-r}^{m+r}g(t)\,dt.
\]
For fixed \(r\), define \(F_r(m):=\int_{m-r}^{m+r}g(t)\,dt\). Since \(g\) is even, \(F_r(-m)=F_r(m)\), while for \(m>0\),
\[
F_r'(m)
=
g(m+r)-g(m-r)
=
g(m+r)-g(|m-r|)>0.
\]
Hence \(F_r\) has its unique minimum at \(m=0\). Therefore
\[
X_I(x)-X_I(y)
\ge
\int_{-r}^{r}g(t)\,dt
=
X_I(-r)-X_I(r)
=
-2X_I(r)
=
2Q_I(r).
\]
Moreover, equality holds if and only if \(m=0\), namely \(x+y=0\). The case \(y<x\) follows by interchanging \(x\) and \(y\).
\end{proof}

\subsection{The sharp radial estimate on balls}

The radial structure of the ball reduces the logarithmic \(p\)-flux to a one-dimensional equation with an additional geometric term. The following lemma records the resulting ODE and its precise behavior near the origin.

\begin{lemma}\label{lem:Q-ode}
Let \(p'=p/(p-1)\). Then
\begin{equation}\label{eq:Q-ode}
Q_B'(r)+\frac{N-1}{r}Q_B(r)
=\lambda_{1,p}(B_R)+(p-1)Q_B(r)^{p'},
\end{equation}
and $\lim_{r\downarrow0}\frac{Q_B(r)}r
=\frac{\lambda_{1,p}(B_R)}{N}.$
More precisely,
\[
Q_B(r)=\frac{\lambda_{1,p}(B_R)}{N}r
+\frac{p-1}{N+p'}\left(\frac{\lambda_{1,p}(B_R)}{N}\right)^{p'}r^{p'+1}
+o(r^{p'+1})
\qquad \text{as}\quad r\downarrow0.
\]
\end{lemma}

\begin{proof}
The radial equation is
\[
-\bigl(r^{N-1}|U_B'|^{p-2}U_B'\bigr)'
=\lambda_{1,p}(B_R) r^{N-1}U_B^{p-1}.
\]
Since \(U_B'=-q_BU_B\) and $|U_B'|^{p-2}U_B'=-Q_BU_B^{p-1},$
division by \(r^{N-1}U_B^{p-1}\) gives \eqref{eq:Q-ode}. Multiplication by \(r^{N-1}\) and integration from \(0\) to \(r\) yield
\[
Q_B(r)=\frac{\lambda_{1,p}(B_R)}{N}r
+(p-1)r^{1-N}\int_0^rs^{N-1}Q_B(s)^{p'}\,ds.
\]
Since \(Q_B(0)=0\) and \(Q_B\) is continuous,
\[
0\le r^{-N}\int_0^r s^{N-1}Q_B(s)^{p'}\,ds
\le \frac1N\sup_{0\le s\le r}Q_B(s)^{p'}\to0,
\]
which gives $\lim_{r\downarrow0}\frac{Q_B(r)}r
=\frac{\lambda_{1,p}(B_R)}N.$
Substituting
\(Q_B(s)=(\lambda_{1,p}(B_R)/N)s+o(s)\)
into the integral identity yields the stated expansion.
\end{proof}

The following lemma records the key monotonicity and convexity properties of the radial logarithmic \(p\)-flux on the ball.

\begin{lemma}\label{lem:Q-convexity}
For \(0<r<R\),
\[
Q_B'(r)>0,
\qquad
H_B(r):=Q_B'(r)-\frac{Q_B(r)}{r}>0,
\]
and
\[
K_B(r):=
Q_B''(r)-\frac{2}{r}
\left(
Q_B'(r)-\frac{Q_B(r)}{r}
\right)>0.
\]
In particular, $Q_B''>0$ and $\left(\frac{Q_B}{r}\right)'>0$ on \((0,R)\).
\end{lemma}

\begin{proof}
Differentiating \eqref{eq:Q-ode} gives
\begin{equation}\label{eq:Q-second}
Q_B''
=
p q_B Q_B'
-\frac{N-1}{r}
\left(
Q_B'-\frac{Q_B}{r}
\right).
\end{equation}
By Lemma~\ref{lem:Q-ode}, \(Q_B'>0\) near \(0\). Suppose that \(Q_B'\) vanishes somewhere, and let \(r_0\) be its first zero. Then \(Q_B'(r)>0\) for \(0<r<r_0\) and \(Q_B'(r_0)=0\), hence necessarily \(Q_B''(r_0)\le0\). On the other hand,
\[
Q_B''(r_0)
=
\frac{N-1}{r_0^2}Q_B(r_0)>0,
\]
which is a contradiction. Therefore \(Q_B'(r)>0\) for all \(0<r<R\). By \eqref{eq:Q-second}, $H_B'+\frac{N}{r}H_B=pq_BQ_B'.$
Thus
\[
r^N H_B(r)
=
p\int_0^r s^N q_B(s)Q_B'(s)\,ds>0.
\]
Let
\[
a:=\frac{\lambda_{1,p}(B_R)}{N},
\qquad
b:=\frac{p-1}{N+p'}a^{p'}.
\]
Using Lemma~\ref{lem:Q-ode} together with \eqref{eq:Q-ode} and \eqref{eq:Q-second}, we obtain
\[
Q_B'(r)
=
a+(p'+1)br^{p'}+o(r^{p'}),
\qquad
Q_B''(r)
=
p'(p'+1)br^{p'-1}+o(r^{p'-1}).
\]
Hence
\[
K_B(r)
=
p'(p'-1)b\,r^{p'-1}
+o(r^{p'-1})>0
\]
for all sufficiently small \(r>0\). Since $K_B=H_B'-\frac{H_B}{r},$
\begin{equation}\label{eq:K-ode}
K_B'
=
\frac{p}{p-1}q_B^{2-p}(Q_B')^2
+p q_B K_B
+\frac{2p}{r}q_BH_B
-\frac{N+1}{r}K_B.
\end{equation}
If \(r_1\) were the first zero of \(K_B\), then \(K_B'(r_1)\le0\), whereas \eqref{eq:K-ode} gives \(K_B'(r_1)>0\). This contradiction proves \(K_B>0\). Finally,
\[
Q_B''
=
K_B+\frac{2H_B}{r}>0,
\qquad
\left(\frac{Q_B}{r}\right)'
=
\frac{H_B}{r}>0,
\]
which completes the proof.
\end{proof}

The preceding properties of \(Q_B\) allow us to derive the corresponding two-point modulus estimate on the ball. The proof reduces the two-point quantity along the chord joining \(x\) and \(y\) to a one-dimensional integral. The positivity of \(H_B\) and \(K_B\) shows that the directional derivative is minimized when the chord passes through the origin, while the strict convexity of \(Q_B\) then implies that, among intervals of fixed length, the minimum occurs when the interval is centered at the origin.

\begin{proposition}\label{prop:intrinsic}
For distinct \(x,y\in B_R\), we have
\begin{equation}\label{eq:F-intrinsic}
\left\langle
X_B(y)-X_B(x),
\frac{y-x}{|y-x|}
\right\rangle
\le
-2Q_B\!\left(\frac{|y-x|}{2}\right).
\end{equation}
Equality in \eqref{eq:F-intrinsic} holds if and only if \(y=-x\).
\end{proposition}

\begin{proof}
Set \(d:=|x-y|\) and \(e:=(y-x)/d\). Write \(x=z+ae\) and \(y=z+be\), where \(z\perp e\) and \(b-a=d\). Let \(\rho:=|z|\) and $r(t):=(\rho^2+t^2)^{1/2}.$
Since
\[
-X_B(z+te)
=
Q_B(r(t))\frac{z+te}{r(t)},
\]
we have
\[
\Psi(t):=
\langle -X_B(z+te),e\rangle
=
Q_B(r(t))\frac{t}{r(t)}.
\]
A direct computation gives
\[
\Psi'(t)
=
\frac{Q_B(r)}{r}
+
H_B(r)\frac{t^2}{r^2}
=:G(r,t).
\]
For fixed \(t\ne0\), Lemma~\ref{lem:Q-convexity} yields
\[
\partial_rG(r,t)
=
\frac{H_B(r)}{r}
\left(1-\frac{t^2}{r^2}\right)
+
K_B(r)\frac{t^2}{r^2}
>0
\qquad\text{when }r>|t|.
\]
Since \(r(t)\ge |t|\), it follows that $\Psi'(t)\ge G(|t|,t)=Q_B'(|t|),$
with equality for all \(t\ne0\) if \(\rho=0\), whereas the inequality is strict for every \(t\ne0\) if \(\rho>0\). Hence
\[
-\left\langle X_B(y)-X_B(x),e\right\rangle
=
\int_a^b\Psi'(t)\,dt
\ge
\int_a^bQ_B'(|t|)\,dt,
\]
and equality here holds if and only if \(\rho=0\). It remains to minimize the last integral among intervals of length \(d\). Write $a=c-\frac d2,\,b=c+\frac d2,$
and define
\[
F_d(c)
:=
\int_{c-d/2}^{c+d/2}Q_B'(|t|)\,dt.
\]
Since \(Q_B''>0\), the function \(t\mapsto Q_B'(|t|)\) is even and strictly increasing in \(|t|\). Therefore
\[
F_d'(c)
=
Q_B'\!\left(\left|c+\frac d2\right|\right)
-
Q_B'\!\left(\left|c-\frac d2\right|\right),
\]
so \(F_d'(c)<0\) for \(c<0\) and \(F_d'(c)>0\) for \(c>0\). Thus \(F_d\) has its unique minimum at \(c=0\), and consequently
\[
\int_a^bQ_B'(|t|)\,dt
\ge
\int_{-d/2}^{d/2}Q_B'(|t|)\,dt
=
2\int_0^{d/2}Q_B'(t)\,dt
=
2Q_B(d/2),
\]
with equality if and only if \(a=-d/2\) and \(b=d/2\). Combining the two inequalities,
\[
-\left\langle X_B(y)-X_B(x),e\right\rangle
\ge
2Q_B(d/2),
\]
which is equivalent to \eqref{eq:F-intrinsic}. Equality holds if and only if \(\rho=0\) and \(a=-b\). Since then \(z=0\), we have \(y=-x\). Conversely, if \(y=-x\), then \(\rho=0\) and \(a=-b\), so equality holds.
\end{proof}

To compare the ball model with its one-dimensional counterpart, we recall the explicit first Dirichlet \(p\)-eigenfunction on \(I=(-R,R)\). Define
\[
\arcsin_p(s):=\int_0^s\frac{d\tau}{(1-\tau^p)^{1/p}},
\qquad
\pi_p:=2\arcsin_p(1),
\]
and let \(\sin_p\) be the inverse of \(\arcsin_p\) on \([0,\pi_p/2]\), extended in the standard way; see \cite{Lindqvist1995,LangEdmunds2011,EdmundsGurkaLang2012}. With the notation of Notation~\ref{not:one-dimensional-model}, $\lambda_{1,p}(I)
=
(p-1)\left(\frac{\pi_p}{2R}\right)^p$ and 
\[
U_I(r)
=
\sin_p\!\left(\frac{\pi_p}{2R}(R-r)\right),
\qquad
0\le r\le R,
\]
Moreover, \(Q_I\) satisfies
\[
Q_I'>0,
\qquad
Q_I''=p\,q_IQ_I'>0,
\qquad
\left(\frac{Q_I}{r}\right)'>0
\quad\text{on }(0,R),
\]
where the last inequality follows from the strict convexity of \(Q_I\) and \(Q_I(0)=0\).

\smallskip

We compare the radial logarithmic \(p\)-flux on the ball with the one-dimensional model through the normalized quantities \(Q_B/r\) and \(Q_I/r\). The argument combines their asymptotic behavior near the origin and the boundary with a contact-point analysis for the corresponding first-order equations. This first yields the eigenvalue comparison \(\lambda_{1,p}(B_R)>N\lambda_{1,p}(I)\), and then rules out any interior contact to obtain \(Q_B>Q_I\) on \((0,R)\).

\begin{proof}[\textbf{Proof of Theorem~\ref{thm:intro-comparison}.}]
Set
\[
w_B(r):=\frac{Q_B(r)}{r},
\qquad
w_I(r):=\frac{Q_I(r)}{r},
\qquad
\delta(r):=w_B(r)-w_I(r).
\]
We first compare the two functions near the endpoints.

\smallskip

Near \(R\), write \(A:=-U_B'(R)>0\). Since \(-U_B'\) stays positive in a
boundary neighborhood, the radial equation is nondegenerate there and may be
written as
\[
-(p-1)(-U_B')^{p-2}U_B''
+\frac{N-1}{r}(-U_B')^{p-1}
=
\lambda_{1,p}(B_R) U_B^{p-1}.
\]
Evaluating at \(r=R\) gives $U_B''(R)=\frac{N-1}{(p-1)R}A.$
Hence, as \(t\downarrow0\),
\[
U_B(R-t)
=
At+\frac{N-1}{2(p-1)R}At^2+o(t^2),
\qquad
-U_B'(R-t)
=
A+\frac{N-1}{(p-1)R}At+o(t),
\]
and therefore
\[
q_B(R-t)
=
\frac1t+\frac{N-1}{2(p-1)R}+o(1).
\]
For the interval case the curvature term is absent; hence
\(U_I''(R)=0\) and $q_I(R-t)=\frac1t+o(1).$
Since \(N\ge2\) and \(q\mapsto q^{p-1}\) is strictly increasing on
\((0,\infty)\), it follows that
\begin{equation}\label{eq:boundary-sign}
\delta(r)>0
\qquad\text{for }r\text{ sufficiently close to }R.
\end{equation}

We next prove $\lambda_{1,p}(B_R)>N\lambda_{1,p}(I).$
Indeed, suppose instead that
\(\lambda_{1,p}(B_R)\le N\lambda_{1,p}(I)\). By
Lemma~\ref{lem:Q-ode},
\[
w_B(r)
=
\frac{\lambda_{1,p}(B_R)}{N}
+
\frac{p-1}{N+p'}
\left(\frac{\lambda_{1,p}(B_R)}{N}\right)^{p'}r^{p'}
+o(r^{p'}),
\]
whereas the one-dimensional equation \eqref{eq:QI-ode-main} gives
\[
w_I(r)
=
\lambda_{1,p}(I)
+
\frac{p-1}{1+p'}\lambda_{1,p}(I)^{p'}r^{p'}
+o(r^{p'}).
\]
If $\lambda_{1,p}(B_R)<N\lambda_{1,p}(I),$
then \(\delta(r)<0\) near \(0\). If $\lambda_{1,p}(B_R)=N\lambda_{1,p}(I),$
then
\[
\delta(r)
=
(p-1)\lambda_{1,p}(I)^{p'}
\left(
\frac1{N+p'}-\frac1{1+p'}
\right)r^{p'}
+o(r^{p'})<0
\]
near \(0\). Thus in either case \(\delta<0\) near \(0\), while
\eqref{eq:boundary-sign} gives \(\delta>0\) near \(R\). Let \((0,b)\) be the
connected component of \(\{\delta<0\}\) containing a neighborhood of \(0\).
Then \(b\in(0,R)\), \(\delta(b)=0\), and \(\delta'(b)\ge0\). From
\eqref{eq:Q-ode} and \eqref{eq:QI-ode-main},
\begin{align}
rw_B'
&=
\lambda_{1,p}(B_R)-Nw_B+(p-1)r^{p'}w_B^{p'},
\label{eq:wB-main}\\
rw_I'
&=
\lambda_{1,p}(I)-w_I+(p-1)r^{p'}w_I^{p'}.
\label{eq:wI-main}
\end{align}
Hence, at every contact point \(r_0\) with
\(w_B(r_0)=w_I(r_0)\),
\begin{equation}\label{eq:contact-main}
r_0\delta'(r_0)
=
\lambda_{1,p}(B_R)-\lambda_{1,p}(I)-(N-1)w_I(r_0).
\end{equation}
Since \(w_I\) is strictly increasing and $\lim_{r\downarrow0}w_I(r)=\lambda_{1,p}(I),$
we have $w_I(b)>\lambda_{1,p}(I).$
Thus, under
\(\lambda_{1,p}(B_R)\le N\lambda_{1,p}(I)\),
\[
b\delta'(b)
\le
(N-1)\lambda_{1,p}(I)-(N-1)w_I(b)<0,
\]
contradicting \(\delta'(b)\ge0\). This proves $\lambda_{1,p}(B_R)>N\lambda_{1,p}(I),$
and consequently
\[
\lim_{r\downarrow0}w_B(r)
=
\frac{\lambda_{1,p}(B_R)}{N}
>
\lambda_{1,p}(I)
=
\lim_{r\downarrow0}w_I(r),
\]
so \(\delta>0\) near \(0\) as well.

It remains to exclude an interior negative region. Suppose that
\(\delta<0\) somewhere, and let \((a,b)\Subset(0,R)\) be a component of
\(\{\delta<0\}\). Then
\[
\delta(a)=\delta(b)=0,
\qquad
\delta'(a)\le0,
\qquad
\delta'(b)\ge0.
\]
Set $T:=
\frac{\lambda_{1,p}(B_R)-\lambda_{1,p}(I)}{N-1}.$
By \eqref{eq:contact-main}, $w_I(a)\ge T$ and $w_I(b)\le T,$
which contradicts the strict increase of \(w_I\) since \(a<b\). Therefore $\delta\ge0$ on $(0,R).$

\smallskip

Finally, suppose that \(\delta(r_0)=0\) at some \(r_0\in(0,R)\). Then \(r_0\) is
a local minimum, so $\delta'(r_0)=0$ and $\delta''(r_0)\ge0.$
Subtracting \eqref{eq:wI-main} from \eqref{eq:wB-main} gives
\[
r\delta'
=
\lambda_{1,p}(B_R)-\lambda_{1,p}(I)
-(N-1)w_I-N\delta
+(p-1)r^{p'}
\bigl((w_I+\delta)^{p'}-w_I^{p'}\bigr).
\]
Differentiating and evaluating at \(r_0\), where $\delta(r_0)=\delta'(r_0)=0,$
yields
\[
r_0\delta''(r_0)
=
-(N-1)w_I'(r_0)<0,
\]
again a contradiction. Hence $w_B(r)>w_I(r)$ for every $0<r<R,$
and Proposition~\ref{prop:intrinsic} yields
\[
\left\langle
X_B(y)-X_B(x),\frac{y-x}{|y-x|}
\right\rangle
\le
-2Q_B\!\left(\frac{|x-y|}{2}\right)
<
-2Q_I\!\left(\frac{|x-y|}{2}\right),
\]
which completes the proof.
\end{proof}

\section{Failure of the One-Dimensional Modulus in Higher Dimensions}

In this section we prove that the one-dimensional logarithmic \(p\)-flux
modulus fails in dimensions \(N\ge2\) for every \(p\neq2\). For
\(1<p<2\), we obtain fixed-scale degeneration of the two-point flux on
smooth uniformly convex domains. For \(p>2\), we use thin domains
\(D\times(-\varepsilon,\varepsilon)\), with \(D\subset\mathbb R^{N-1}\)
bounded and convex, together with the rescaled limit of the first
\(p\)-eigenfunction and the resulting tangential flux limit.

\subsection{The subquadratic case \texorpdfstring{$1<p<2$}{1<p<2}}

We show that, for every sufficiently small prescribed scale, the
two-point flux quantity can be made arbitrarily small on smooth uniformly
convex domains with fixed diameter and fixed point separation. The proof
starts from a convex domain with a flat boundary patch, where the
tangential logarithmic \(p\)-flux tends to zero near the boundary, and
then passes to smooth uniformly convex approximations while preserving
the diameter and the prescribed separation.
\begin{proof}[\textbf{Proof of Theorem~\ref{thm:failure-universal-two-point-subquadratic}.}]

\textbf{Step 1. Flat boundary degeneration.}

Let \(\Omega_0\subset\mathbb R^N\) be a bounded \(C^\infty\) convex domain
whose boundary contains a relatively open flat patch. After a dilation, we
may assume that $\operatorname{diam}(\Omega_0)=D.$
Let \(u_0>0\) be the first eigenfunction of \eqref{eq:mainequation},
normalized by $\|u_0\|_{L^p(\Omega_0)}=1.$
By choosing suitable Euclidean coordinates, there exist an open set
\(U\subset\mathbb R^{N-1}\) and \(\varepsilon_0>0\) such that
\[
\{(s,0):s\in U\}\subset\partial\Omega_0,
\qquad
U\times(0,\varepsilon_0)\subset\Omega_0.
\]
Fix \(U_0\Subset U\). By the boundary \(C^{1,\alpha}\)-regularity and the
Hopf boundary lemma
\cite[Theorem~1]{Lieberman1988}\cite[Theorem~5]{Vazquez1984},
\[
a(s):=\partial_tu_0(s,0)>0,
\qquad
s\in U_0.
\]
After shrinking the neighborhood, \(|\nabla u_0|\ge c_0>0\); hence the
\(p\)-Laplace equation is uniformly elliptic there. Indeed, wherever
\(\nabla u_0\neq0\),
\[
-\left(
\delta_{ij}
+(p-2)\frac{(u_0)_i(u_0)_j}{|\nabla u_0|^2}
\right)(u_0)_{ij}
=
\lambda_{1,p}(\Omega_0)u_0^{p-1}|\nabla u_0|^{2-p},
\]
and the coefficient matrix has eigenvalues \(1\) and \(p-1\). Local
boundary Schauder estimates therefore yield \(u_0\in C^{2,\beta}\) for
some \(\beta\in(0,1)\); see \cite{Lieberman1988} and
\cite[Corollary~6.7 and Theorem~6.19]{GilbargTrudinger}. Thus, for every
tangential unit vector \(e\),
\[
u_0(s,t)=a(s)t+O(t^2),
\qquad
\partial_eu_0(s,t)
=
t\,\partial_ea(s)+O(t^{1+\beta}).
\]
Writing \(v_0:=\log u_0\), we obtain
\[
\partial_tv_0=\frac1t+O(1),
\qquad
\partial_ev_0
=
\partial_e\log a(s)+O(t^\beta)
=
O(1).
\]
Hence $|\nabla v_0|^{p-2}
=
t^{2-p}(1+O(t)),$
and therefore, uniformly for \(s\in U_0\),
\begin{equation}\label{eq:flat-tangential-vanishing}
\left\langle X_{\Omega_0}(s,t),e\right\rangle
=
O(t^{2-p})
\longrightarrow0
\quad\text{as }t\downarrow0,\quad X_{\Omega_0}
:=
|\nabla\log u_0|^{p-2}\nabla\log u_0.
\end{equation}

Choose \(s_*\in U_0\), a tangential unit vector \(e_1\in\mathbb R^{N-1}\),
and \(r_*>0\) such that $s_*+\rho e_1\in U_0$ for every $|\rho|\le r_*.$
Fix any \(r\in(0,r_*)\), and set $e:=(e_1,0)\in\mathbb R^N.$
For \(t\in(0,\varepsilon_0)\), define
\[
x_t:=(s_*-re_1,t),
\qquad
y_t:=(s_*+re_1,t).
\]
Then by \eqref{eq:flat-tangential-vanishing},
\begin{equation}\label{eq:flat-fixed-scale-vanishing}
\left\langle
X_{\Omega_0}(y_t)-X_{\Omega_0}(x_t),e
\right\rangle
\longrightarrow0
\qquad\text{as }t\downarrow0.
\end{equation}

\smallskip

\textbf{Step 2. Uniformly convex approximation.}

After a translation, assume \(0\in\operatorname{int}\Omega_0\), and let
\[
\mu(x):=\inf\{\rho>0:x\in \rho\Omega_0\}
\]
be the Minkowski functional of \(\Omega_0\). Then \(\mu\) is convex and
positively homogeneous of degree one,
\[
\Omega_0=\{\mu<1\},
\qquad
\partial\Omega_0=\{\mu=1\},
\]
and \(\mu\) is \(C^\infty\) in a neighborhood of \(\partial\Omega_0\);
see \cite[Chapters~1--2]{Schneider2014}. For \(\varepsilon>0\), set
\[
F_\varepsilon(x):=\mu(x)+\varepsilon|x|^2,
\qquad
\widetilde\Omega_\varepsilon:=\{F_\varepsilon<1\}.
\]
For all sufficiently small \(\varepsilon\), \(\partial\widetilde\Omega_\varepsilon\)
lies in the neighborhood where \(\mu\) is smooth. Since \(\mu\) is convex,
\[
D^2F_\varepsilon
=
D^2\mu+2\varepsilon I
\ge2\varepsilon I.
\]
Moreover, on \(\partial\widetilde\Omega_\varepsilon\), Euler's identity
\(x\cdot\nabla\mu=\mu\) gives
\[
x\cdot\nabla F_\varepsilon
=
\mu(x)+2\varepsilon|x|^2
=
1+\varepsilon|x|^2>0,
\]
so \(\nabla F_\varepsilon\neq0\) and
\(\partial\widetilde\Omega_\varepsilon\) is \(C^\infty\). If
\(\tau\in T_x\partial\widetilde\Omega_\varepsilon\), then
\[
\mathrm{II}_{\partial\widetilde\Omega_\varepsilon}(\tau,\tau)
=
\frac{D^2F_\varepsilon(x)[\tau,\tau]}
{|\nabla F_\varepsilon(x)|}
\ge
\frac{2\varepsilon}
{\max_{\partial\widetilde\Omega_\varepsilon}|\nabla F_\varepsilon|}
|\tau|^2,
\]
hence \(\widetilde\Omega_\varepsilon\) is uniformly convex. Furthermore,
\[
\widetilde\Omega_\varepsilon\subset\Omega_0,
\qquad
\widetilde\Omega_\varepsilon\uparrow\Omega_0
\quad\text{as }\varepsilon\downarrow0.
\]
Writing $D_\varepsilon:=\operatorname{diam}(\widetilde\Omega_\varepsilon),\,
a_\varepsilon:=\frac{D}{D_\varepsilon},\,
\Omega_\varepsilon:=a_\varepsilon\widetilde\Omega_\varepsilon,$
we have
\[
D_\varepsilon\longrightarrow D,
\qquad
a_\varepsilon\longrightarrow1,
\qquad
\operatorname{diam}(\Omega_\varepsilon)=D.
\]

Let \(\widetilde u_\varepsilon>0\) be the first \(p\)-eigenfunction on
\(\widetilde\Omega_\varepsilon\), normalized by $\|\widetilde u_\varepsilon\|_{L^p(\widetilde\Omega_\varepsilon)}=1.$
Since \(\widetilde\Omega_\varepsilon\subset\Omega_0\), $\lambda_{1,p}(\widetilde\Omega_\varepsilon)
\ge
\lambda_{1,p}(\Omega_0).$
Conversely, for every \(\varphi\in C_c^\infty(\Omega_0)\), one has
\(\operatorname{supp}\varphi\subset\widetilde\Omega_\varepsilon\) for all
sufficiently small \(\varepsilon\), and therefore
\[
\limsup_{\varepsilon\downarrow0}
\lambda_{1,p}(\widetilde\Omega_\varepsilon)
\le
\frac{\int_{\Omega_0}|\nabla\varphi|^p\,dx}
{\int_{\Omega_0}|\varphi|^p\,dx}.
\]
Taking the infimum over
\(\varphi\in C_c^\infty(\Omega_0)\setminus\{0\}\) yields $\lambda_{1,p}(\widetilde\Omega_\varepsilon)
\rightarrow
\lambda_{1,p}(\Omega_0).$
Extend \(\widetilde u_\varepsilon\) by zero to \(\Omega_0\). Then
\(\{\widetilde u_\varepsilon\}\) is bounded in
\(W_0^{1,p}(\Omega_0)\), and, after passing to a subsequence,
\[
\widetilde u_\varepsilon\rightharpoonup u
\quad\text{in }W_0^{1,p}(\Omega_0),
\qquad
\widetilde u_\varepsilon\to u
\quad\text{in }L^p(\Omega_0).
\]
Thus \(\|u\|_{L^p(\Omega_0)}=1\), and weak lower semicontinuity gives
\[
\int_{\Omega_0}|\nabla u|^p\,dx
\le
\liminf_{\varepsilon\downarrow0}
\lambda_{1,p}(\widetilde\Omega_\varepsilon)
=
\lambda_{1,p}(\Omega_0).
\]
On the other hand, the variational characterization gives $\lambda_{1,p}(\Omega_0)
\le
\int_{\Omega_0}|\nabla u|^p\,dx.$
Hence equality holds, so \(u\) is a normalized first eigenfunction of
\(\Omega_0\). By simplicity, $u=u_0.$

Let \(K\Subset\Omega_0\) be compact. For all sufficiently small
\(\varepsilon\), \(K\Subset\widetilde\Omega_\varepsilon\). Since
\(\lambda_{1,p}(\widetilde\Omega_\varepsilon)\) is bounded, the standard
local boundedness estimates and the interior \(C^{1,\gamma}\)-regularity
for \(p\)-Laplace equations \cite{Tolksdorf1984} yield, for some
\(\gamma\in(0,1)\), $\sup_{\varepsilon\ll1}
\|\widetilde u_\varepsilon\|_{C^{1,\gamma}(K)}
<\infty.$
By compactness and uniqueness of the limit, $\widetilde u_\varepsilon\rightarrow u_0$ in $C^1(K).$ Since \(u_0>0\) on \(K\),
\[
\nabla\log\widetilde u_\varepsilon
\longrightarrow
\nabla\log u_0
\qquad\text{uniformly on }K.
\]
As \(z\mapsto |z|^{p-2}z\) is continuous on \(\mathbb R^N\) for every
\(p>1\),
\begin{equation}\label{eq:flux-domain-convergence}
X_{\widetilde\Omega_\varepsilon}
\longrightarrow
X_{\Omega_0}
\quad\text{uniformly on }K,\quad X_{\widetilde\Omega_\varepsilon}
:=
|\nabla\log\widetilde u_\varepsilon|^{p-2}
\nabla\log\widetilde u_\varepsilon.
\end{equation}

Choose any sequence \(t_k\downarrow0\). Since \(a_\varepsilon\to1\), for
each fixed \(k\) and all sufficiently small \(\varepsilon\), the points
\[
\widetilde x_{\varepsilon,k}
:=
\left(s_*-\frac{r}{a_\varepsilon}e_1,t_k\right),
\qquad
\widetilde y_{\varepsilon,k}
:=
\left(s_*+\frac{r}{a_\varepsilon}e_1,t_k\right)
\]
belong to \(\widetilde\Omega_\varepsilon\).
For each \(k\), choose \(\varepsilon_k>0\) sufficiently small that
\[
\left|
\left\langle
X_{\widetilde\Omega_{\varepsilon_k}}
(\widetilde y_{\varepsilon_k,k})
-
X_{\widetilde\Omega_{\varepsilon_k}}
(\widetilde x_{\varepsilon_k,k}),
e
\right\rangle
-
\left\langle
X_{\Omega_0}(s_*+re_1,t_k)
-
X_{\Omega_0}(s_*-re_1,t_k),
e
\right\rangle
\right|
\le\frac1k.
\]
Here we have used \eqref{eq:flux-domain-convergence}, together with
\(a_{\varepsilon_k}\to1\) and the continuity of \(X_{\Omega_0}\) at the
fixed height \(t_k>0\). By \eqref{eq:flat-fixed-scale-vanishing},
\[
\left\langle
X_{\widetilde\Omega_{\varepsilon_k}}
(\widetilde y_{\varepsilon_k,k})
-
X_{\widetilde\Omega_{\varepsilon_k}}
(\widetilde x_{\varepsilon_k,k}),
e
\right\rangle
\longrightarrow0.
\]

Set $a_k:=a_{\varepsilon_k},\,
\Omega_k:=a_k\widetilde\Omega_{\varepsilon_k},\,
x_k:=a_k\widetilde x_{\varepsilon_k,k},\,
y_k:=a_k\widetilde y_{\varepsilon_k,k}.$
Then
\[
\operatorname{diam}(\Omega_k)=D,
\qquad
|x_k-y_k|=2r,
\qquad
\frac{y_k-x_k}{|y_k-x_k|}=e.
\]
If \(u_k\) denotes the first eigenfunction on \(\Omega_k\), then, up to
multiplication by a positive normalization constant, $u_k(x)
=
\widetilde u_{\varepsilon_k}(x/a_k).$
Consequently,
\[
\begin{aligned}
\left\langle
X_{\Omega_k}(y_k)-X_{\Omega_k}(x_k),e
\right\rangle
&=
a_k^{1-p}
\left\langle
X_{\widetilde\Omega_{\varepsilon_k}}
(\widetilde y_{\varepsilon_k,k})
-
X_{\widetilde\Omega_{\varepsilon_k}}
(\widetilde x_{\varepsilon_k,k}),
e
\right\rangle\longrightarrow0,
\end{aligned}
\]
which completes the proof.
\end{proof}

\subsection{The superquadratic case and thin-domain asymptotics}

We begin with the thin-domain asymptotics for the first eigenfunction.
After rescaling the transverse variable, the second-order term in the
\(p\)-energy reduces, after a suitable change of variables, to the
Dirichlet energy on an arbitrary bounded convex domain
\(D\subset\mathbb R^{N-1}\). The rescaled eigenfunctions converge to the
product of the first one-dimensional \(p\)-eigenfunction in the
transverse variable and a power of the first Dirichlet eigenfunction of
the Laplacian on \(D\). Near a maximum point of the latter, the limiting
tangential logarithmic \(p\)-flux is of order \(t^{p-1}\), whereas the
one-dimensional comparison term is of order \(t\). Since \(p>2\), this
gives the violation in
Theorem~\ref{thm:general-convex-flux-failure-p-superquadratic}.

We first establish several auxiliary lemmas needed for the thin-domain
analysis.

\begin{lemma}
\label{lem:thin-algebraic-coercivity}
Let \(p>2\). Then there exists \(c_p>0\) such that, for all
\(a,c\in\mathbb R\) and \(b\in\mathbb R^{N-1}\),
\begin{align}
\bigl(|b|^2+(a+c)^2\bigr)^{p/2}
-|a|^p-p|a|^{p-2}ac
&\geq
\frac p2|a|^{p-2}|b|^2,
\label{eq:algebra-longitudinal}\\
\bigl(|b|^2+(a+c)^2\bigr)^{p/2}
-|a|^p-p|a|^{p-2}ac
&\geq
c_p|a|^{p-2}c^2.
\label{eq:algebra-transverse}
\end{align}
\end{lemma}

\begin{proof}
For \(a\neq0\), after setting
\[
s:=\frac{|b|}{|a|},
\qquad
y:=\frac{a+c}{a},
\]
\eqref{eq:algebra-longitudinal} reduces to
\[
F(s,y)
:=
(s^2+y^2)^{p/2}
-\frac p2s^2-py+p-1
\geq0.
\]
Since \(p>2\), \(F\) is coercive, and its critical point equations are
\[
s\bigl((s^2+y^2)^{(p-2)/2}-1\bigr)=0,
\qquad
y(s^2+y^2)^{(p-2)/2}=1.
\]
Their unique solution is \((s,y)=(0,1)\), where \(F(0,1)=0\), proving
\eqref{eq:algebra-longitudinal}. Moreover, strict convexity of
\(t\mapsto|t|^p\), together with
\[
\lim_{r\to0}
\frac{|1+r|^p-1-pr}{r^2}
=
\frac{p(p-1)}2,
\qquad
\frac{|1+r|^p-1-pr}{r^2}
\longrightarrow+\infty
\quad\text{as }|r|\to\infty,
\]
gives $\inf_{r\in\mathbb R}
\frac{|1+r|^p-1-pr}{r^2}>0,$
where the quotient at \(r=0\) is defined by continuity. Since
\[
\bigl(|b|^2+(a+c)^2\bigr)^{p/2}
\geq |a+c|^p,
\]
this proves \eqref{eq:algebra-transverse}. The case \(a=0\) follows
directly.
\end{proof}

The next lemma combines the Picone decomposition with the preceding
algebraic estimates to obtain coercive control of both the longitudinal
and transverse derivatives of \(H\).

\begin{lemma}
\label{lem:thin-picone-coercivity}
Let \(p>2\) and \(D\subset\mathbb R^{N-1}\) be a bounded domain,
\(J:=(-1,1)\), \(Q:=D\times J\), and let \(\phi>0\) be the first
Dirichlet \(p\)-eigenfunction on \(J\), normalized by
\(\int_J\phi^p\,dz=1\). Set $\rho(z):=|\phi'(z)|^{p-2}\phi(z)^2.$
For \(\varepsilon>0\), define
\[
E_\varepsilon(v)
:=
\int_Q
\bigl(|v_z|^2+\varepsilon^2|\nabla_xv|^2\bigr)^{p/2}
\,dx\,dz.
\]
If \(v\in W_0^{1,p}(Q)\) is nonnegative and $H:=\left(\frac v\phi\right)^{p/2},$
then
\begin{align}
E_\varepsilon(v)-\lambda_{1,p}(J)\int_Qv^p
&\geq
\frac2p\,\varepsilon^2
\int_Q\rho|\nabla_xH|^2,
\label{eq:picone-longitudinal-bound}\\
E_\varepsilon(v)-\lambda_{1,p}(J)\int_Qv^p
&\geq
c_p
\int_Q\rho|H_z|^2,
\label{eq:picone-transverse-bound}
\end{align}
where \(c_p>0\) depends only on \(p\).
\end{lemma}

\begin{proof}
Set \(D_\varepsilon:=(\varepsilon\nabla_x,\partial_z)\). For
\(\delta>0\), define
\[
\phi_\delta:=\phi+\delta,
\qquad
h_\delta:=\frac{v}{\phi_\delta},
\qquad
H_\delta:=h_\delta^{p/2},
\]
and $\xi_\delta:=(0,h_\delta\phi'),\,
\eta_\delta
:=
(\varepsilon\phi_\delta\nabla_xh_\delta,
\phi_\delta h_{\delta,z}).$
Since \(v=h_\delta\phi_\delta\) and \(\phi_\delta'=\phi'\),
\[
D_\varepsilon v
=
\xi_\delta+\eta_\delta.
\]
The following computation may first be carried out for bounded
\(v\in W_0^{1,p}(Q)\); the general case follows by truncation. Testing
\[
-\bigl(|\phi'|^{p-2}\phi'\bigr)'
=
\lambda_{1,p}(J)\phi^{p-1}
\]
with $\frac{v^p}{\phi_\delta^{p-1}}
=
h_\delta^p\phi_\delta$
and integrating also in \(x\), we obtain
\[
\lambda_{1,p}(J)
\int_Q
\phi^{p-1}\frac{v^p}{\phi_\delta^{p-1}}
=
\int_Q
\Bigl(
|\xi_\delta|^p
+
p|\xi_\delta|^{p-2}\xi_\delta\cdot\eta_\delta
\Bigr).
\]
Hence
\begin{equation}\label{eq:picone-regularized}
E_\varepsilon(v)
-
\lambda_{1,p}(J)
\int_Q
\phi^{p-1}\frac{v^p}{\phi_\delta^{p-1}}
=
\int_Q
\Bigl(
|\xi_\delta+\eta_\delta|^p
-
|\xi_\delta|^p
-
p|\xi_\delta|^{p-2}\xi_\delta\cdot\eta_\delta
\Bigr).
\end{equation}
Compare with the \(p\)-Picone identity of
\cite[Theorem~1.1]{AllegrettoHuang1998}.

Applying Lemma~\ref{lem:thin-algebraic-coercivity} to
\eqref{eq:picone-regularized}, with
\[
a=h_\delta\phi',
\qquad
b=\varepsilon\phi_\delta\nabla_xh_\delta,
\qquad
c=\phi_\delta h_{\delta,z},
\]
and using
\[
h_\delta^{p-2}|\nabla_xh_\delta|^2
=
\frac4{p^2}|\nabla_xH_\delta|^2,
\qquad
h_\delta^{p-2}h_{\delta,z}^2
=
\frac4{p^2}|H_{\delta,z}|^2,
\]
we obtain, after decreasing \(c_p\) if necessary,
\begin{align*}
E_\varepsilon(v)
-
\lambda_{1,p}(J)
\int_Q
\phi^{p-1}\frac{v^p}{\phi_\delta^{p-1}}
&\geq
\frac2p\,\varepsilon^2
\int_Q
|\phi'|^{p-2}\phi_\delta^2
|\nabla_xH_\delta|^2,\\
E_\varepsilon(v)
-
\lambda_{1,p}(J)
\int_Q
\phi^{p-1}\frac{v^p}{\phi_\delta^{p-1}}
&\geq
c_p
\int_Q
|\phi'|^{p-2}\phi_\delta^2
|H_{\delta,z}|^2.
\end{align*}
For every \(0<\eta<1\), \(\phi\) is bounded away from zero on
\(J_\eta:=(-1+\eta,1-\eta)\); hence, with $h:=\frac v\phi$ and $H:=h^{p/2},$
one has $h\in W^{1,p}(D\times J_\eta),\,
H\in W^{1,2}(D\times J_\eta),$
and
\[
\nabla_xH
=
\frac p2h^{(p-2)/2}\nabla_xh,
\qquad
H_z
=
\frac p2h^{(p-2)/2}h_z
\]
a.e. on \(D\times J_\eta\). Since $0\leq\phi^{p-1}\phi_\delta^{1-p}\leq1,$
dominated convergence on the left, Fatou's lemma on
\(D\times J_\eta\), and then \(\eta\downarrow0\) yield
\eqref{eq:picone-longitudinal-bound} and
\eqref{eq:picone-transverse-bound}.
\end{proof}

The following weighted Poincar\'e estimate controls the transverse
oscillation relative to the weight \(\rho\).

\begin{lemma}
\label{lem:weighted-transverse-poincare}
Let \(p>2\), \(J:=(-1,1)\), and let \(\phi>0\) be the first Dirichlet
eigenfunction corresponding to \(\lambda_{1,p}(J)\), normalized by $\int_J\phi^p\,dz=1.$
Set
\[
\rho(z):=|\phi'(z)|^{p-2}\phi(z)^2,
\qquad
A_p:=\int_J\rho(z)\,dz.
\]
For every locally absolutely continuous \(\psi\) satisfying $\int_J\rho|\psi'|^2\,dz<\infty,$
set $\overline\psi_\rho
:=
\frac1{A_p}\int_J\rho\psi\,dz.$
Then
\begin{equation}\label{eq:weighted-poincare-rho}
\int_J\phi^p|\psi-\overline\psi_\rho|^2\,dz
+
|\psi(0)-\overline\psi_\rho|^2
\leq
C_p\int_J\rho|\psi'|^2\,dz.
\end{equation}
\end{lemma}

\begin{proof}
Since \(\phi\) is even and strictly decreasing on \((0,1)\),
integration of its eigenvalue equation gives
\[
(-\phi')^{p-1}(z)
=
\lambda_{1,p}(J)
\int_0^z\phi(s)^{p-1}\,ds,
\qquad
0<z<1.
\]
Hence
\[
\rho(z)\asymp |z|^{(p-2)/(p-1)}
\quad\text{as }z\to0,
\qquad
\rho(z)\asymp(1-|z|)^2,
\quad
\phi(z)^p\asymp(1-|z|)^p
\quad\text{as }|z|\to1.
\]
Define $K(z)
:=
\left|
\int_0^z\frac{ds}{\rho(s)}
\right|.$
Then
\[
K(z)\asymp |z|^{1/(p-1)}
\quad\text{as }z\to0,
\qquad
K(z)\asymp(1-|z|)^{-1}
\quad\text{as }|z|\to1,
\]
so that \(\rho K,\phi^pK\in L^1(J)\). By Cauchy--Schwarz,
\[
|\psi(z)-\psi(0)|^2
\leq
K(z)\int_J\rho|\psi'|^2.
\]
Therefore
\[
|\overline\psi_\rho-\psi(0)|^2
\leq
\frac1{A_p}
\int_J\rho|\psi-\psi(0)|^2\,dz
\leq
C_p\int_J\rho|\psi'|^2\,dz,
\]
and consequently
\[
\begin{aligned}
\int_J\phi^p|\psi-\overline\psi_\rho|^2\,dz
&\leq
2\int_J\phi^p|\psi-\psi(0)|^2\,dz
+
2|\psi(0)-\overline\psi_\rho|^2\leq
C_p\int_J\rho|\psi'|^2\,dz.
\end{aligned}
\]
This proves \eqref{eq:weighted-poincare-rho}.
\end{proof}

To pass from the \(L^2\)-limit on the central section to local uniform
convergence, we use the following compactness property of log-concave
functions.

\begin{lemma}
\label{lem:logconcave-L2-uniform}
Let \(U\subset\mathbb R^m\) be an open convex set and let
\(f_j:U\to(0,\infty)\) be log-concave. If $f_j\rightarrow f$ in $L^2_{\mathrm{loc}}(U),$
where \(f>0\) is continuous, then
\[
f_j\longrightarrow f
\qquad\text{locally uniformly in }U.
\]
\end{lemma}

\begin{proof}
It suffices to show that every subsequence admits a further subsequence
converging locally uniformly to \(f\). Starting from an arbitrary
subsequence, the \(L^2_{\mathrm{loc}}(U)\)-convergence and a diagonal
argument yield a further subsequence, still denoted by \(f_j\), such that
\[
f_j(x)\longrightarrow f(x)
\qquad\text{for a.e. }x\in U.
\]
Let \(E\subset U\) denote the full-measure set of convergence; in
particular, \(E\) is dense in \(U\). Set
\[
g_j:=-\log f_j,
\qquad
g:=-\log f.
\]
Then \(g_j\) is finite and convex on \(U\), \(g\) is continuous, and $g_j(x)\rightarrow g(x)$ for every $x\in E.$

Fix \(K\Subset U\). Choose a compact convex polytope \(P\) such that $K\Subset\operatorname{int}P\Subset U$
and whose vertices belong to \(E\). Since \(g_j\) converges at the
vertices of \(P\), convexity gives a uniform upper bound for \(g_j\) on
\(P\). Choose \(x_*\in E\cap\operatorname{int}P\). Since
\(K\Subset\operatorname{int}P\), there exists \(\theta\in(0,1)\) such
that, for every \(x\in K\), one can find \(y\in P\) satisfying $x_*=\theta x+(1-\theta)y.$
Hence
\[
g_j(x_*)
\leq
\theta g_j(x)+(1-\theta)g_j(y),
\]
which, together with the uniform upper bound on \(P\) and the boundedness
of \(g_j(x_*)\), gives a uniform lower bound for \(g_j\) on \(K\).
Applying the same argument on a compact set $K\Subset\operatorname{int}K'\Subset\operatorname{int}P$
yields a uniform bound for \(|g_j|\) on \(K'\). The interior slope
estimate for convex functions then gives a uniform Lipschitz bound for
\(g_j\) on \(K\).

By Arzel\`a--Ascoli, the present subsequence admits a further subsequence
converging uniformly on \(K\). Its limit agrees with \(g\) on the dense
set \(E\cap K\), and hence on all of \(K\). Since the original
subsequence was arbitrary, $g_j\rightarrow g$ uniformly on $K.$
As \(K\Subset U\) is arbitrary, \(g_j\to g\) locally uniformly in \(U\).
Therefore
\[
f_j=e^{-g_j}\longrightarrow e^{-g}=f
\]
locally uniformly in \(U\).
\end{proof}

We now prove Theorem~\ref{profileprop:thin-domain} in three steps.

\begin{proof}[\textbf{Proof of Theorem~\ref{profileprop:thin-domain}.}]
\medskip
\noindent\textbf{Step 1. Rescaling and compactness.}
Set
\[
Q:=D\times J,
\qquad
w_\varepsilon(x,z):=u_\varepsilon(x,\varepsilon z),
\qquad
v_\varepsilon
:=
\frac{w_\varepsilon}{\|w_\varepsilon\|_{L^p(Q)}}.
\]
Then \(\|v_\varepsilon\|_{L^p(Q)}=1\), and, writing $\mu_\varepsilon
:=
\varepsilon^p\lambda_{1,p}(\Omega_\varepsilon),$
the change of variables \(y=\varepsilon z\) gives
\begin{equation}\label{eq:thin-rayleigh}
\mu_\varepsilon
=
\min_{\substack{v\in W_0^{1,p}(Q)\\ \|v\|_{L^p(Q)}=1}}
E_\varepsilon(v),
\qquad
E_\varepsilon(v)
:=
\int_Q
\bigl(|v_z|^2+\varepsilon^2|\nabla_xv|^2\bigr)^{p/2}
\,dx\,dz.
\end{equation}

Taking \(f(x)\phi(z)\), where
\(f\in C_c^\infty(D)\), \(f\geq0\), and \(\int_Df^p=1\), as a test
function and using
\[
(a^2+t^2)^{p/2}
\leq
a^p+C_p\bigl(a^{p-2}t^2+|t|^p\bigr),
\qquad a,t\geq0,
\]
we obtain
\begin{equation}\label{eq:thin-energy-Oeps2}
0
\leq
\mu_\varepsilon-\lambda_{1,p}(J)
\leq
C\varepsilon^2.
\end{equation}
Set $H_\varepsilon
:=
\left(\frac{v_\varepsilon}{\phi}\right)^{p/2}.$
Lemma~\ref{lem:thin-picone-coercivity} and
\eqref{eq:thin-energy-Oeps2} imply
\begin{equation}\label{eq:H-bounds}
\int_Q\rho|\nabla_xH_\varepsilon|^2\leq C,
\qquad
\int_Q\rho|H_{\varepsilon,z}|^2\leq C\varepsilon^2.
\end{equation}

For a.e. \(z\in J\),
\(v_\varepsilon(\cdot,z)\in W_0^{1,p}(D)\). Since \(\phi(z)>0\), $h_\varepsilon(\cdot,z)
:=
\frac{v_\varepsilon(\cdot,z)}{\phi(z)}
\in W_0^{1,p}(D).$
The chain rule and H\"older's inequality give
\[
\begin{aligned}
\int_D|\nabla_xH_\varepsilon|^2\,dx
&=
\frac{p^2}{4}
\int_Dh_\varepsilon^{p-2}|\nabla_xh_\varepsilon|^2\,dx\\
&\leq
\frac{p^2}{4}
\left(\int_Dh_\varepsilon^p\,dx\right)^{(p-2)/p}
\left(\int_D|\nabla_xh_\varepsilon|^p\,dx\right)^{2/p},
\end{aligned}
\]
so \(H_\varepsilon(\cdot,z)\in H_0^1(D)\) for a.e. \(z\). The
Poincar\'e inequality on \(D\), followed by \eqref{eq:H-bounds}, yields
\[
\int_Q\rho H_\varepsilon^2
\leq
\frac1{\lambda_{1,2}(D)}
\int_Q\rho|\nabla_xH_\varepsilon|^2
\leq C.
\]
Define $G_\varepsilon(x)
:=
\frac1{A_p}
\int_J\rho(z)H_\varepsilon(x,z)\,dz.$
Then \(G_\varepsilon\in H_0^1(D)\), with
\[
\nabla G_\varepsilon
=
\frac1{A_p}\int_J\rho\nabla_xH_\varepsilon\,dz,
\qquad
A_p|\nabla G_\varepsilon(x)|^2
\leq
\int_J\rho|\nabla_xH_\varepsilon|^2\,dz.
\]
Applying Lemma~\ref{lem:weighted-transverse-poincare} to
\(H_\varepsilon(x,\cdot)\) for a.e. \(x\) and integrating in \(x\), we
obtain
\begin{equation}\label{eq:H-to-G}
\int_Q\phi^p|H_\varepsilon-G_\varepsilon|^2\,dx\,dz
\leq
C\varepsilon^2,
\qquad
\int_D|H_\varepsilon(x,0)-G_\varepsilon(x)|^2\,dx
\leq
C\varepsilon^2.
\end{equation}
Since $1
=
\int_Qv_\varepsilon^p
=
\int_Q\phi^pH_\varepsilon^2$
and \(\int_J\phi^p=1\), \eqref{eq:H-to-G} gives
\[
\left|
\|G_\varepsilon\|_{L^2(D)}-1
\right|
\leq
\|G_\varepsilon-H_\varepsilon\|_{L^2(Q,\phi^p)}
\longrightarrow0.
\]
Thus
\begin{equation}\label{eq:G-L2-normalization}
\int_DG_\varepsilon^2\,dx\longrightarrow1.
\end{equation}

\medskip
\noindent\textbf{Step 2. Second-order asymptotics.}
By Lemma~\ref{lem:thin-picone-coercivity} and the definition of
\(G_\varepsilon\),
\[
\frac{\mu_\varepsilon-\lambda_{1,p}(J)}{\varepsilon^2}
\geq
\frac2p
\int_Q\rho|\nabla_xH_\varepsilon|^2
\geq
\frac{2A_p}{p}
\int_D|\nabla G_\varepsilon|^2.
\]
Since \(G_\varepsilon\in H_0^1(D)\), the sharp Dirichlet Poincar\'e
inequality gives
\begin{equation}\label{eq:thin-second-order-liminf}
\liminf_{\varepsilon\downarrow0}
\frac{\mu_\varepsilon-\lambda_{1,p}(J)}{\varepsilon^2}
\geq
\frac{2A_p}{p}\lambda_{1,2}(D).
\end{equation}

For the matching upper bound, choose nonnegative
\(G_k\in C_c^\infty(D)\), vanishing to infinite order at the boundary of
their supports, such that
\[
\int_DG_k^2\,dx=1,
\qquad
G_k\longrightarrow G_D
\quad\text{in }H_0^1(D),
\]
and set $f_k:=G_k^{2/p},\,
v_k(x,z):=f_k(x)\phi(z).$
Then \(f_k\in W_0^{1,p}(D)\) and \(\|v_k\|_{L^p(Q)}=1\). For fixed
\(k\),
\[
\frac{
\bigl(
f_k^2|\phi'|^2
+
\varepsilon^2|\nabla f_k|^2\phi^2
\bigr)^{p/2}
-
f_k^p|\phi'|^p
}{\varepsilon^2}
\longrightarrow
\frac p2
f_k^{p-2}|\nabla f_k|^2
|\phi'|^{p-2}\phi^2
\]
a.e. in \(Q\). Moreover, for \(0<\varepsilon<1\),
\[
\begin{aligned}
0
&\leq
\frac{
\bigl(
f_k^2|\phi'|^2
+
\varepsilon^2|\nabla f_k|^2\phi^2
\bigr)^{p/2}
-
f_k^p|\phi'|^p
}{\varepsilon^2}\leq
C_p\left(
f_k^{p-2}|\nabla f_k|^2
|\phi'|^{p-2}\phi^2
+
|\nabla f_k|^p\phi^p
\right),
\end{aligned}
\]
and the right-hand side is integrable. Hence dominated convergence gives
\[
\begin{aligned}
\lim_{\varepsilon\downarrow0}
\frac{E_\varepsilon(v_k)-\lambda_{1,p}(J)}{\varepsilon^2}
&=
\frac p2
\left(
\int_J|\phi'|^{p-2}\phi^2\,dz
\right)
\int_Df_k^{p-2}|\nabla f_k|^2\,dx=
\frac{2A_p}{p}
\int_D|\nabla G_k|^2\,dx,
\end{aligned}
\]
where $f_k^{p-2}|\nabla f_k|^2
=
\frac4{p^2}|\nabla G_k|^2.$
By the minimality of \(v_\varepsilon\) and then \(k\to\infty\),
\begin{equation}\label{eq:thin-second-order-limsup}
\limsup_{\varepsilon\downarrow0}
\frac{\mu_\varepsilon-\lambda_{1,p}(J)}{\varepsilon^2}
\leq
\frac{2A_p}{p}\lambda_{1,2}(D).
\end{equation}
Combining \eqref{eq:thin-second-order-liminf} and
\eqref{eq:thin-second-order-limsup} proves
\eqref{eq:thin-eigenvalue-second-order}.

Moreover, the preceding lower bound and
\eqref{eq:G-L2-normalization} imply $\int_D|\nabla G_\varepsilon|^2\,dx
\rightarrow
\lambda_{1,2}(D).$
Since \(\{G_\varepsilon\}\) is bounded in \(H_0^1(D)\), every sequence
\(\varepsilon_j\downarrow0\) admits a subsequence such that
\[
G_{\varepsilon_j}\rightharpoonup G
\quad\text{in }H_0^1(D),
\qquad
G_{\varepsilon_j}\longrightarrow G
\quad\text{in }L^2(D).
\]
Then \(G\geq0\), \(\|G\|_{L^2(D)}=1\), and $\int_D|\nabla G|^2\,dx
\leq
\lambda_{1,2}(D).$
The sharp Poincar\'e inequality gives the reverse inequality. Hence
\(G=G_D\), and the convergence of the \(L^2\)-norms of the gradients
gives
\begin{equation}\label{eq:G-strong}
G_\varepsilon
\longrightarrow
G_D
\qquad\text{strongly in }H_0^1(D).
\end{equation}

\medskip
\noindent\textbf{Step 3. Identification of the rescaled limit.}
By \eqref{eq:H-to-G} and \eqref{eq:G-strong},
\[
\int_Q
\phi^p|H_\varepsilon-G_D|^2\,dx\,dz
\longrightarrow0.
\]
Since $v_\varepsilon
=
\phi H_\varepsilon^{2/p},$
and \(2/p\in(0,1)\), we have
\[
|a^{2/p}-b^{2/p}|
\leq
|a-b|^{2/p},
\qquad a,b\geq0.
\]
Therefore
\[
\begin{aligned}
\int_Q
\left|
v_\varepsilon-\phi G_D^{2/p}
\right|^p\,dx\,dz
&=
\int_Q
\phi^p
\left|
H_\varepsilon^{2/p}-G_D^{2/p}
\right|^p\,dx\,dz\\
&\leq
\int_Q
\phi^p|H_\varepsilon-G_D|^2\,dx\,dz
\longrightarrow0.
\end{aligned}
\]
Thus $v_\varepsilon
\rightarrow
\phi G_D^{2/p}$ in $L^p(Q).$
The positive first \(p\)-eigenfunction is log-concave on bounded convex
domains \cite[Theorem~1.4]{ChenHauer2026}. Since $v_\varepsilon(x,z)
=
\frac{u_\varepsilon(x,\varepsilon z)}
{\|w_\varepsilon\|_{L^p(Q)}},$
the function \(v_\varepsilon\) is positive and log-concave on \(Q\).
Since \(p>2\), $v_\varepsilon
\rightarrow
\phi G_D^{2/p}$ in $L^2_{\mathrm{loc}}(Q),$
Lemma~\ref{lem:logconcave-L2-uniform} therefore gives
\[
v_\varepsilon(x,z)
\longrightarrow
\phi(z)G_D(x)^{2/p}
\]
locally uniformly in \(D\times J\). Finally, $u_\varepsilon(x,\varepsilon z)
=
\frac{v_\varepsilon(x,z)}{v_\varepsilon(x_0,0)}.$
Hence
\[
u_\varepsilon(x,\varepsilon z)
\longrightarrow
\frac{\phi(z)}{\phi(0)}
\left(
\frac{G_D(x)}
{G_D(x_0)}
\right)^{2/p}
\]
locally uniformly in \(D\times J\). This completes the proof.
\end{proof}

We now turn to the proof of Theorem~\ref{thm:general-convex-flux-failure-p-superquadratic}.

\begin{proof}[\textbf{Proof of Theorem~\ref{thm:general-convex-flux-failure-p-superquadratic}.}]
For \(0<\varepsilon<1\), let
\[
\Omega_\varepsilon
:=
D\times(-\varepsilon,\varepsilon),
\qquad
2R_\varepsilon
:=
\operatorname{diam}(\Omega_\varepsilon)
=
\sqrt{\operatorname{diam}(D)^2+4\varepsilon^2}.
\]
Set $R_D:=\frac{\operatorname{diam}(D)}2,$
so that \(R_\varepsilon\to R_D\) as \(\varepsilon\downarrow0\).

Let \(G_D>0\) be the first Dirichlet eigenfunction of the Laplacian on
\(D\), and choose \(x_D\in D\) such that $G_D(x_D)=\max_D G_D.$
Let \(u_\varepsilon>0\) be the first eigenfunction of
\eqref{eq:mainequation} on \(\Omega_\varepsilon\), normalized by
\(u_\varepsilon(x_D,0)=1\). By
Theorem~\ref{profileprop:thin-domain},
\begin{equation}\label{eq:thin-profile-general-D}
u_\varepsilon(x,0)
\longrightarrow
F_D(x)
:=
\left(
\frac{G_D(x)}{G_D(x_D)}
\right)^{2/p}
\end{equation}
locally uniformly in \(D\).

Set
\[
h_\varepsilon(x):=\log u_\varepsilon(x,0),
\qquad
h(x):=\log F_D(x)
=
\frac2p
\log\frac{G_D(x)}{G_D(x_D)}.
\]
The positive first \(p\)-eigenfunction is log-concave on bounded convex
domains \cite[Theorem~1.4]{ChenHauer2026}; hence \(h_\varepsilon\) is
concave on \(D\). Since \(F_D>0\) in \(D\),
\eqref{eq:thin-profile-general-D} yields $h_\varepsilon\longrightarrow h$ locally uniformly in $D.$
It follows that
\begin{equation}\label{eq:gradient-convergence-central-section}
\nabla h_\varepsilon(x)
\longrightarrow
\nabla h(x)
\qquad\text{for every }x\in D.
\end{equation}
Indeed, if \(e\in\mathbb S^{N-2}\) and \(s>0\) is such that
\(x\pm se\in D\), concavity gives
\[
\frac{h_\varepsilon(x+se)-h_\varepsilon(x)}{s}
\leq
\partial_eh_\varepsilon(x)
\leq
\frac{h_\varepsilon(x)-h_\varepsilon(x-se)}{s}.
\]
Passing first to the limit \(\varepsilon\downarrow0\) and then
\(s\downarrow0\), and using \(h\in C^\infty(D)\), proves
\eqref{eq:gradient-convergence-central-section}.

Since \(\Omega_\varepsilon\) is invariant under \(z\mapsto-z\), simplicity
of the first eigenvalue implies that \(u_\varepsilon\) is even in \(z\).
Consequently, $\partial_z u_\varepsilon(x,0)=0.$
Thus, with
\[
Y_D(x)
:=
|\nabla h(x)|^{p-2}\nabla h(x),
\]
\eqref{eq:gradient-convergence-central-section} and the continuity of
\(\xi\mapsto|\xi|^{p-2}\xi\) give
\begin{equation}\label{eq:central-flux-limit-general-D}
X_{\Omega_\varepsilon}(x,0)
\longrightarrow
\bigl(Y_D(x),0\bigr)
\qquad\text{for every }x\in D.
\end{equation}

Fix the unit vector \(e\in\mathbb S^{N-2}\) in the statement. Since
\(x_D\in D\), there exists \(t_0>0\) such that
$x_D+te\in D$ for \(0<t<t_0\).
Since \(x_D\) is a maximum point of \(G_D\), $\nabla h(x_D)=0,\,
Y_D(x_D)=0.$
Moreover, \(h\in C^\infty(D)\), and hence
$\nabla h(x_D+te)=O(t)$ as $t\downarrow0.$
Therefore
\begin{equation}\label{eq:general-D-flux-small-t}
\left|
\left\langle
Y_D(x_D+te),e
\right\rangle
\right|
\leq
C t^{p-1}
\qquad\text{for }0<t<t_0.
\end{equation}

On the other hand, writing \(p':=p/(p-1)\), the one-dimensional model
satisfies
\[
Q_{I_{R_D}}'(s)
=
\lambda_{1,p}(I_{R_D})
+
(p-1)Q_{I_{R_D}}(s)^{p'},
\qquad
Q_{I_{R_D}}(0)=0.
\]
Thus
\begin{equation}\label{eq:one-dimensional-Q-small-t}
2Q_{I_{R_D}}\!\left(\frac t2\right)
=
\lambda_{1,p}(I_{R_D})t+o(t).
\end{equation}
Since \(p>2\), \eqref{eq:general-D-flux-small-t} and
\eqref{eq:one-dimensional-Q-small-t} imply that there exists
\(t_{p,D}\in(0,t_0)\) such that, for every \(t\in(0,t_{p,D})\),
\begin{equation}\label{eq:strict-limit-flux-violation-general-D}
\left\langle
Y_D(x_D+te)-Y_D(x_D),e
\right\rangle
>
-2Q_{I_{R_D}}\!\left(\frac t2\right).
\end{equation}

Fix \(t\in(0,t_{p,D})\). By
\eqref{eq:central-flux-limit-general-D},
\[
\left\langle
X_{\Omega_\varepsilon}(x_D+te,0)
-
X_{\Omega_\varepsilon}(x_D,0),
e
\right\rangle
\longrightarrow
\left\langle
Y_D(x_D+te)-Y_D(x_D),e
\right\rangle.
\]
Moreover, the scaling $Q_{I_R}(s)
=
R^{1-p}Q_{I_1}(s/R)$
and \(R_\varepsilon\to R_D\) give
\[
Q_{I_\varepsilon}\!\left(\frac t2\right)
\longrightarrow
Q_{I_{R_D}}\!\left(\frac t2\right).
\]
Hence, by the strict inequality
\eqref{eq:strict-limit-flux-violation-general-D}, there exists
\(\varepsilon_{p,D,t}>0\) such that, for every
\(0<\varepsilon<\varepsilon_{p,D,t}\),
\[
\left\langle
X_{\Omega_\varepsilon}(x_D+te,0)
-
X_{\Omega_\varepsilon}(x_D,0),
e
\right\rangle
>
-2Q_{I_\varepsilon}\!\left(\frac t2\right).
\]
Taking $x:=(x_D,0),\,
y:=(x_D+te,0),$
we have \(|x-y|=t\) and $\frac{y-x}{|y-x|}=(e,0).$
Therefore
\[
\left\langle
X_{\Omega_\varepsilon}(y)-X_{\Omega_\varepsilon}(x),
\frac{y-x}{|y-x|}
\right\rangle
>
-2Q_{I_\varepsilon}\!\left(\frac{|x-y|}{2}\right),
\]
which proves the theorem.
\end{proof}

\section{The Logarithmic \texorpdfstring{$p$}{p}-Flux and \texorpdfstring{$p=2$}{p=2} Rigidity}
\label{sec:geometric-structure-logarithmic-flux}

This section compares the Hessian of an arbitrary \(C^2\) function \(w\)
with the symmetric differential of its \(p\)-gradient. We show that
positive semidefiniteness of the latter implies convexity of \(w\), while
the converse holds for all \(w\) if and only if \(p=2\). The result is
then applied to \(w=-\log u\).

\smallskip

We first prove the level-set decomposition stated in
Proposition~\ref{prop:level-set-flux-decomposition}.

\begin{proof}[\textbf{Proof of Proposition~\ref{prop:level-set-flux-decomposition}.}]
Since \(\nabla w=s\nu\), for
\(\tau,\eta\in T_x\Sigma_x\),
\[
D^2w(\tau,\eta)
=
\langle D_\tau(s\nu),\eta\rangle
=
s\,\langle D_\tau\nu,\eta\rangle
=
s\,\mathrm{II}(\tau,\eta).
\]
Moreover, $D^2w(\tau,\nu)
=
\langle D_\tau(s\nu),\nu\rangle
=
\partial_\tau s,$
because \(\langle D_\tau\nu,\nu\rangle=0\), while
\[
D^2w(\nu,\nu)
=
\langle D_\nu\nabla w,\nu\rangle
=
\partial_\nu s.
\]
Since $\partial_\tau s
=
\langle\nabla_\Sigma s,\tau\rangle,$
this gives
\[
D^2w=
\begin{pmatrix}
s\,\mathrm{II} & \nabla_\Sigma s\\
(\nabla_\Sigma s)^T & \partial_\nu s
\end{pmatrix}.
\]

Next, from $A=s^{p-2}\nabla w$
we obtain
\[
DA
=
s^{p-2}D^2w
+
(p-2)s^{p-3}\nabla w\otimes\nabla s.
\]
Therefore
\[
M_p[w]
=
s^{p-2}D^2w
+
\frac{p-2}{2}s^{p-3}
\bigl(
\nabla w\otimes\nabla s
+
\nabla s\otimes\nabla w
\bigr).
\]
Using $\nabla w=s\nu$ and $\nabla s
=
\nabla_\Sigma s+(\partial_\nu s)\nu,$
the tangential--tangential block is \(s^{p-1}\mathrm{II}\), the
tangential--normal block is
\[
s^{p-2}
\left(
1+\frac{p-2}{2}
\right)\nabla_\Sigma s
=
\frac p2s^{p-2}\nabla_\Sigma s,
\]
and the normal--normal entry is $(p-1)s^{p-2}\partial_\nu s.$
Hence
\[
M_p[w]
=
s^{p-2}
\begin{pmatrix}
s\,\mathrm{II} & \dfrac p2\nabla_\Sigma s\\[2mm]
\dfrac p2(\nabla_\Sigma s)^T & (p-1)\partial_\nu s
\end{pmatrix},
\]
which completes the proof.
\end{proof}

The decomposition separates the second fundamental form of the level sets, the normal variation of \(s=|\nabla w|\), and its tangential variation \(\nabla_\Sigma s\). The latter is absent in one dimension and vanishes whenever \(s\) is constant along the level sets.

We next record the proof of Corollary~\ref{cor:rigid-flux-cases}.

\begin{proof}[\textbf{Proof of Corollary~\ref{cor:rigid-flux-cases}.}]
If \(p=2\), then \(A=\nabla w\), and hence $M_2[w]=D^2w.$
If \(N=1\), then $A=|w'|^{p-2}w',$
so
\[
M_p[w]=A'
=
(p-1)|w'|^{p-2}w''.
\]
Since \(p>1\) and \(w'\neq0\) at regular points, the factor
\((p-1)|w'|^{p-2}\) is positive, which proves
\[
M_p[w]\geq0
\quad\Longleftrightarrow\quad
w''\geq0.
\]

Finally, if \(\nabla_\Sigma s=0\), then
Proposition~\ref{prop:level-set-flux-decomposition} gives
\[
D^2w=
\begin{pmatrix}
s\,\mathrm{II} & 0\\
0 & \partial_\nu s
\end{pmatrix},
\qquad
M_p[w]
=
s^{p-2}
\begin{pmatrix}
s\,\mathrm{II} & 0\\
0 & (p-1)\partial_\nu s
\end{pmatrix}.
\]
Since \(s>0\) on \(\mathcal R_w\) and \(p-1>0\), the two block-diagonal
matrices have the same sign. Hence
\[
D^2w\geq0
\quad\Longleftrightarrow\quad
M_p[w]\geq0,
\]
which completes the proof.
\end{proof}

We now turn to the proof of Theorem~\ref{thm:p2-flux-rigidity}.

\begin{proof}[\textbf{Proof of Theorem~\ref{thm:p2-flux-rigidity}.}]
Fix a point \(x\) with \(\nabla w(x)\neq0\), and write, with respect to
\(T_x\Sigma_x\oplus\operatorname{span}\{\nu\}\),
\[
D^2w=
\begin{pmatrix}
B & b\\
b^T & c
\end{pmatrix},
\qquad
s^{2-p}M_p[w]
=
\begin{pmatrix}
B & \dfrac p2 b\\[2mm]
\dfrac p2 b^T & (p-1)c
\end{pmatrix},
\]
where $B:=s\,\mathrm{II},\,
b:=\nabla_\Sigma s,\,
c:=\partial_\nu s.$

Assume first that \(M_p[w](x)\geq0\). Since \(s>0\), the second matrix
above is positive semidefinite, and in particular \(c\geq0\). Hence, for
every \(\xi\in T_x\Sigma_x\) and \(t\in\mathbb R\),
\[
0
\leq
\begin{pmatrix}\xi\\[1mm] \dfrac{2t}{p}\end{pmatrix}^{\!T}
s^{2-p}M_p[w]
\begin{pmatrix}\xi\\[1mm] \dfrac{2t}{p}\end{pmatrix}
=
\langle B\xi,\xi\rangle
+2t\langle b,\xi\rangle
+\frac{4(p-1)}{p^2}ct^2.
\]
Since $\frac{4(p-1)}{p^2}\le 1$
and \(c\geq0\), it follows that
\[
\langle B\xi,\xi\rangle
+2t\langle b,\xi\rangle
+ct^2
\geq0.
\]
Thus \(D^2w(x)\geq0\). Therefore $M_p[w](x)\geq0\Longrightarrow
D^2w(x)\geq0$
for every \(p>1\).

If \(p=2\), then $M_2[w]=D^2w,$
so the two positivity conditions are equivalent. It remains to show that the converse fails for every \(p\neq2\). Set $\kappa_p
:=
\frac{p^2}{4(p-1)}>1.$
Choose \(b\in\mathbb R^{N-1}\setminus\{0\}\) and \(c\) such that $|b|^2<c<\kappa_p|b|^2,$
and define
\[
w(x',x_N)
=
x_N
+\frac12|x'|^2
+x_N\,b\cdot x'
+\frac c2x_N^2.
\]
Then
\[
\nabla w(0)=e_N,
\qquad
D^2w=
\begin{pmatrix}
I_{N-1} & b\\
b^T & c
\end{pmatrix}.
\]
Since \(c>|b|^2\), the Schur complement gives $D^2w>0$, see, for instance,
\cite[Chapter~1]{Zhang2005}.
At the origin, $s=1,\,
\mathrm{II}=I_{N-1},\,
\nabla_\Sigma s=b,\,
\partial_\nu s=c,$
and Proposition~\ref{prop:level-set-flux-decomposition} yields
\[
M_p[w](0)
=
\begin{pmatrix}
I_{N-1} & \dfrac p2 b\\[2mm]
\dfrac p2 b^T & (p-1)c
\end{pmatrix}.
\]
Its Schur complement with respect to \(I_{N-1}\) is
\[
(p-1)c-\frac{p^2}{4}|b|^2
=
(p-1)\bigl(c-\kappa_p|b|^2\bigr)<0.
\]
Since \(I_{N-1}\) is positive definite, \(M_p[w](0)\) has a negative
eigenvalue. Hence $D^2w(0)>0$
does not imply \(M_p[w](0)\geq0\) when \(p\neq2\). Therefore the two
positivity conditions are equivalent for every \(w\) and every point
with \(\nabla w\neq0\) if and only if \(p=2\).
\end{proof}

\smallskip


\noindent\textbf{Acknowledgements.}
The author would like to express their sincere gratitude to Professors Bobo Hua and Daniel Hauer for their valuable guidance, helpful suggestions, and continuous encouragement throughout the development of this work.

\medskip

\noindent\textbf{Conflict of interest.} The author declares that they have no conflict of interest.

\medskip

\noindent\textbf{Data availability.} No datasets were generated or analyzed during the current study.

\medskip

\noindent\textbf{AI assistance statement.}
The author used OpenAI's GPT-5.6 Sol model to assist with language polishing and to suggest ideas related to the analysis in Section~4. All mathematical arguments, proofs, and verifications were carried out independently by the authors, who take full responsibility for the content of the paper.

\bigskip

\small
	
	\noindent\textit{Rui Chen}: School of Mathematical Sciences, Fudan University,\\[1mm]
		Shanghai 200433,  China\\[2mm]
		Brandenburg University of Technology Cottbus--Senftenberg,\\[1mm]
		Cottbus 03046, Germany\\[1mm]
		\noindent\emph{Email:} \texttt{chenrui23@m.fudan.edu.cn}\\[3mm]

\vspace{1em}


\medskip

\begin{thebibliography}{99}

\bibitem{AghajaniMoslehTehrani2018}
A.~Aghajani and A.~Mosleh Tehrani,
Lower bounds for the principal eigenvalue of the $p$-Laplacian on the unit ball,
\emph{Numer. Funct. Anal. Optim.}
\textbf{39} (2018), no.~13, 1440--1465.

\bibitem{AllegrettoHuang1998}
W.~Allegretto and Y.~X. Huang,
A Picone's identity for the $p$-Laplacian and applications,
\emph{Nonlinear Anal.}
\textbf{32} (1998), no.~7, 819--830.

\bibitem{Anane1987}
A.~Anane,
Simplicit\'e et isolation de la premi\`ere valeur propre du $p$-laplacien avec poids,
\emph{C. R. Acad. Sci. Paris S\'er. I Math.}
\textbf{305} (1987), no.~16, 725--728.

\bibitem{AndrewsClutterbuck2011}
B.~Andrews and J.~Clutterbuck,
Proof of the fundamental gap conjecture,
\emph{J. Amer. Math. Soc.}
\textbf{24} (2011), no.~3, 899--916.


\bibitem{Benedikt2015}
J.~Benedikt,
Estimates of the principal eigenvalue of the $p$-Laplacian and the
$p$-biharmonic operator,
\emph{Math. Bohem.}
\textbf{140} (2015), no.~2, 215--222.

\bibitem{BenediktDrabek2012}
J.~Benedikt and P.~Dr\'abek,
Estimates of the principal eigenvalue of the $p$-Laplacian,
\emph{J. Math. Anal. Appl.}
\textbf{393} (2012), no.~1, 311--315.

\bibitem{Bhattacharya1988}
T.~Bhattacharya,
Radial symmetry of the first eigenfunction for the $p$-Laplacian in the ball,
\emph{Proc. Amer. Math. Soc.}
\textbf{104} (1988), no.~1, 169--174.

\bibitem{BorrelliMosconiSquassina2024}
W.~Borrelli, S.~Mosconi and M.~Squassina,
Concavity properties for solutions to $p$-Laplace equations with
concave nonlinearities,
\emph{Adv. Calc. Var.}
\textbf{17} (2024), no.~1, 79--97.

\bibitem{BrascampLieb1976}
H.~J. Brascamp and E.~H. Lieb,
On extensions of the Brunn--Minkowski and Pr\'ekopa--Leindler
theorems, including inequalities for log concave functions, and with
an application to the diffusion equation,
\emph{J. Funct. Anal.}
\textbf{22} (1976), no.~4, 366--389.

\bibitem{BryanClutterbuckRankin2026}
P.~Bryan, J.~Clutterbuck and C.~Rankin,
Convexity inequalities for eigenvalues and log-concavity of
eigenfunctions,
arXiv:2605.01334, 2026.

\bibitem{ChenHauer2026}
R.~Chen and D.~Hauer,
Fundamental gaps for the Dirichlet \(p\)-Laplacian with convex potentials:
sharp one-dimensional bounds and a higher-dimensional dichotomy,
\emph{arXiv preprint arXiv:2608.13443} (2026).

\bibitem{CaffarelliFriedman1985}
L.~A. Caffarelli and A.~Friedman,
Convexity of solutions of semilinear elliptic equations,
\emph{Duke Math. J.}
\textbf{52} (1985), no.~2, 431--456.

\bibitem{CaffarelliSpruck1982}
L.~A. Caffarelli and J.~Spruck,
Convexity properties of solutions to some classical variational
problems,
\emph{Comm. Partial Differential Equations}
\textbf{7} (1982), no.~11, 1337--1379.

\bibitem{ColesantiQinSalani2026}
A.~Colesanti, L.~Qin and P.~Salani,
Geometric properties of solutions to elliptic PDE's in Gauss space
and related Brunn--Minkowski type inequalities,
\emph{Adv. Math.}
\textbf{489} (2026), 110827.

\bibitem{CrastaFragala2020}
G.~Crasta and I.~Fragal\`a,
The Brunn--Minkowski inequality for the principal eigenvalue of fully
nonlinear homogeneous elliptic operators,
\emph{Adv. Math.}
\textbf{359} (2020), 106855.

\bibitem{EdmundsGurkaLang2012}
D.~E. Edmunds, P.~Gurka and J.~Lang,
Properties of generalized trigonometric functions,
\emph{J. Approx. Theory}
\textbf{164} (2012), no.~1, 47--56.


\bibitem{GilbargTrudinger}
D.~Gilbarg and N.~S. Trudinger,
\emph{Elliptic Partial Differential Equations of Second Order},
Classics in Mathematics,
Springer-Verlag, Berlin, 2001.



\bibitem{Kajikiya2015}
R.~Kajikiya,
A priori estimate for the first eigenvalue of the $p$-Laplacian,
\emph{Differential Integral Equations}
\textbf{28} (2015), no.~9--10, 1011--1028.

\bibitem{KajikiyaTakeuchi2025}
R.~Kajikiya and S.~Takeuchi,
Estimates for the first eigenvalue of the one-dimensional $p$-Laplacian,
\emph{Appl. Anal.}
\textbf{104} (2025), no.~15, 3041--3053.

\bibitem{KawohlNovaga2008}
B.~Kawohl and M.~Novaga,
The $p$-Laplace eigenvalue problem as $p$ approaches $1$ and
Cheeger sets in a Finsler metric,
\emph{J. Convex Anal.}
\textbf{15} (2008), no.~3, 623--634.

\bibitem{Korevaar1983}
N.~J. Korevaar,
Convex solutions to nonlinear elliptic and parabolic boundary value
problems,
\emph{Indiana Univ. Math. J.}
\textbf{32} (1983), no.~4, 603--614.

\bibitem{LangEdmunds2011}
J.~Lang and D.~E. Edmunds,
\emph{Eigenvalues, Embeddings and Generalised Trigonometric Functions},
Lecture Notes in Mathematics, vol.~2016,
Springer-Verlag, Berlin and Heidelberg, 2011.

\bibitem{Lieberman1988}
G.~M. Lieberman,
Boundary regularity for solutions of degenerate elliptic equations,
\emph{Nonlinear Anal.}
\textbf{12} (1988), no.~11, 1203--1219.

\bibitem{Lindqvist1990}
P.~Lindqvist,
On the equation
$\operatorname{div}(|\nabla u|^{p-2}\nabla u)
+\lambda|u|^{p-2}u=0$,
\emph{Proc. Amer. Math. Soc.}
\textbf{109} (1990), no.~1, 157--164.

\bibitem{Lindqvist1995}
P.~Lindqvist,
Some remarkable sine and cosine functions,
\emph{Ricerche Mat.}
\textbf{44} (1995), no.~2, 269--290.

\bibitem{Lorch1993}
L.~Lorch,
Some inequalities for the first positive zeros of Bessel functions,
\emph{SIAM J. Math. Anal.}
\textbf{24} (1993), no.~3, 814--823.

\bibitem{Ni2013}
L.~Ni,
Estimates on the modulus of expansion for vector fields solving
nonlinear equations,
\emph{J. Math. Pures Appl. (9)}
\textbf{99} (2013), no.~1, 1--16.

\bibitem{otani-teshima}
M.~Otani and T.~Teshima,
On the first eigenvalue of some quasilinear elliptic equations,
\emph{Proc. Japan Acad. Ser. A Math. Sci.}
\textbf{64} (1988), no.~1, 8--10.

\bibitem{Qin2024}
L.~Qin,
Log-concavity of eigenfunction and Brunn--Minkowski inequality of
eigenvalue for weighted $p$-Laplace operator,
arXiv:2411.16377, 2024.

\bibitem{Sakaguchi1987}
S.~Sakaguchi,
Concavity properties of solutions to some degenerate quasilinear
elliptic Dirichlet problems,
\emph{Ann. Scuola Norm. Sup. Pisa Cl. Sci. (4)}
\textbf{14} (1987), no.~3, 403--421.

\bibitem{Schneider2014}
R.~Schneider,
\emph{Convex Bodies: The Brunn--Minkowski Theory},
2nd expanded ed.,
Encyclopedia of Mathematics and its Applications, vol.~151,
Cambridge University Press, Cambridge, 2014.


\bibitem{Tolksdorf1984}
P.~Tolksdorf,
Regularity for a more general class of quasilinear elliptic equations,
\emph{J. Differential Equations}
\textbf{51} (1984), no.~1, 126--150.

\bibitem{Vazquez1984}
J.~L. V\'azquez,
A strong maximum principle for some quasilinear elliptic equations,
\emph{Appl. Math. Optim.}
\textbf{12} (1984), no.~3, 191--202.

\bibitem{Zhang2005}
F.~Zhang, ed.,
\emph{The Schur Complement and Its Applications},
Numerical Methods and Algorithms, vol.~4,
Springer, New York, 2005.

\end{thebibliography}
 \end{document}